\documentclass[11pt]{amsart}
\usepackage{amsmath, palatino, mathpazo, amsfonts, amssymb, mathtools}
\usepackage[mathscr]{eucal}
\usepackage[all]{xy}
\usepackage{datetime}

\usepackage{amsthm}
\usepackage {tikz}
\usepackage {tgbonum}                 
\usepackage {amsmath}                   
\usepackage {amssymb}                   
\usepackage {bibnames}                  
\usepackage {cite}                      
\usepackage {url}                       
\usepackage {varioref}                  

\usepackage{tabularx}
\usepackage{array}

\newtheorem{theorem}{Theorem}[section]
\newtheorem{lemma}[theorem]{Lemma}
\newtheorem{proposition}[theorem]{Proposition}
\newtheorem{corollary}[theorem]{Corollary}
\newtheorem{conjecture}[theorem]{Conjecture}
\newtheorem*{main theorem*}{Main Theorem}

\theoremstyle{remark}
\newtheorem{remark}[theorem]{Remark}

\newtheorem{definition}[theorem]{Definition}

\numberwithin{equation}{section}

\newcommand{\op}{\operatorname}

\newcommand{\M}{\overline{\mathcal{M}}}
\newcommand{\K}{\mathcal{K}}

\newcommand{\so}{\mathscr{O}}

\newcommand{\vir}{{\rm vir}}
\newcommand{\ev}{{\rm ev}}
\newcommand{\Res}{{\rm Res}}
\newcommand{\td}{{\rm td}}
\newcommand{\ch}{{\rm ch}}

\newcommand{\Tr}{{\rm Tr}}

\newcommand{\fake}{{\rm fake}}

\newcommand{\KK}{{K}}
\newcommand{\Coeff}{{\rm Coeff}}
\newcommand{\leg}{{\rm leg}}
\newcommand{\arm}{{\rm arm}}
\newcommand{\tail}{{\rm tail}}

\def\O{\mathcal{O}}

\newcommand{\GW}{{\mathrm{GW}}}
\newcommand{\GV}{{\mathrm{GV}}}

\newcommand{\e}{\epsilon}

\newcommand{\ind}{{\mathrm{ind}}}
	
\title[QK $=$ GV on CY3 II]{Quantum $K$-invariants and Gopakumar--Vafa invariants II. \\ Calabi-Yau threefolds at genus zero}
	
	\author{Y.-C.~Chou}
	\email{bensonchou72@gmail.com, bensonchou@gate.sinica.edu.tw}
	
	\author{Y.-P.~Lee}
	\email{yplee@math.utah.edu, ypleemath@gate.sinica.edu.tw}
	
\address{Institute of Mathematics, Academia Sinica, Taipei 106319, Taiwan, and
Department of Mathematics, University of Utah, 	Salt Lake City, Utah 84112-0090, U.S.A.}

\begin{document}

\maketitle

\begin{abstract}
This is the second part of our ongoing project on the relations between \emph{Gopakumar--Vafa BPS invariants} (GV) and \emph{quantum $K$-theory} (QK) on the Calabi--Yau threefolds (CY3). 
We show that on CY3 a genus zero quantum $K$-invariant can be written as a linear combination of a finite number of Gopakumar--Vafa invariants with coefficients from an explicit ``multiple cover formula''. Conversely, GV can be determined by QK in a similar manner. The technical heart is a proof of a remarkable conjecture by Hans Jockers and Peter Mayr.

This result is consistent with the ``virtual Clemens conjecture'' for the Calabi--Yau threefolds. A heuristic derivation of the relation between QK and GV via the virtual Clemens conjecture and the multiple cover formula is also given.
\end{abstract}

\tableofcontents

\setcounter{section}{-1}

\section{Introduction}

\subsection{Main result}
This paper continues our study of relations between \emph{Gopakumar--Vafa invariants} \cite{Gopakumar_Vafa_1998} and \emph{quantum $K$-theory} \cite{Givental_2000, Lee_2004} on Calabi--Yau threefolds. 
The main result of this paper is the following theorem.

\begin{main theorem*}
A genus zero quantum $K$-invariant (with ``descendants'') can be written as an linear combination of a finite number of Gopakumar--Vafa invariants with coefficients from an explicit multiple cover formula. Conversely, all genus zero Gopakumar--Vafa invariants can be determined by a finite number of quantum $K$-invariants in a similar manner. Furthermore, the linear transformations between QK and GV preserve integrality.
\end{main theorem*}

The multiple cover formula is stated in Proposition~\ref{prop_multicover_t}. The explicit formulation can be found in Theorem~\ref{conjecture_npt} together with the proof of Proposition~\ref{p:1.5}. The quantum $K$-invariants of any genus are intrisically integral invariants, while the integrality of the GV in our definition follows from the result in \cite{Ionel_Parker_2018}.

This theorem generalizes the results in \cite{Chou_Lee_2022_QKGVI} in two directions. First, the equivalence between genus zero quantum $K$-invariants and Gopakumar--Vafa invariants now holds for all Calabi--Yau threefolds. Second, the relation extends to $n$-pointed quantum $K$-invariants thanks to the conjectural formula of H.~Jockers and P.~Mayr. The proof also follows a different approach, although both ultimately come from the virtual Hirzebruch--Riemann--Roch theorem for stacks \cite{Tonita_VKHRR}.

Our techniques generalizes to higher dimensions. The Gopakumar-Vafa type invariants at genus zero have been defined to all semi-positive (including Calabi-Yau and Fano) varieties with dimensions greater than or equal to $3$. For $\gamma_1,\dots, \gamma_l \in H^*(X,\mathbb{Z})$, the numbers $\GV_{0,\beta}(\gamma_1,\dots,\gamma_l)$ are defined by
\[
    \sum_{\beta \neq 0} \GW_{0,\beta}(\gamma_1,\dots,\gamma_l) q^{\beta} =: \sum_{\beta\neq 0}\GV_{0,\beta}(\gamma_1,\dots,\gamma_l) \sum_{d=1}^{\infty} \frac{1}{d^{3-l}} q^{d\beta}
\]
and conjectured to be integers in \cite{Klemm_Pandharipande_CY4_2008}. The integrality has been proved by E.N.\ Ionel and T.H.\ Parker in \cite{Ionel_Parker_2018} using symplectic geometry. As an application of techniques developed in this series of papers an \emph{alternative proof of integrality} is given in \cite{Chou_2023}, by relating QK and GV.

\subsection{Enumerative invariants on Calabi--Yau threefolds}

It is our belief that there should be only one set of numerical (virtual) enumerative invariants on Calabi--Yau threefolds, and Gopakumar--Vafa invariants should serve as the basic invariants to which all others can be reduced. While very explicit (conjectural, and partially verified) relations between Gromov--Witten, Donaldson--Thomas, Pandharipande--Thomas and Gopakumar--Vafa invariants have been available from the beginning, the relation with the quantum $K$-theory was not as clear at first. The situation changed when A.~Givental and his collaborators made a breakthrough in relating quantum $K$-theory with the Gromov--Witten theory \cite{Givental_Tonita_2011, Givental_PEK}. However, that relation is very general and is applicable to any smooth projective varieties (or DM stacks). Hence it is also quite complicated and its implications in this special case of Calabi--Yau threefolds are not as clear. This is the goal we set for ourselves in this series of papers.

As in Part I \cite{Chou_Lee_2022_QKGVI}, we use the \textit{ad hoc} definition of Gopakumar--Vafa invariants in terms of Gromov--Witten invariants via the formula of by R.~Gopakumar and C.~Vafa \cite{Gopakumar_Vafa_1998}.
\begin{equation} \label{e:0.1}
\begin{split}
&\sum _{g=0}^{\infty }~\sum _{\beta \in H_{2}(M,\mathbb {Z} )}{\GW}_{g,\beta} q^{\beta }\lambda ^{2g-2} \\
= & \sum _{g=0}^{\infty }~\sum _{k=1}^{\infty }~\sum _{\beta \in H_{2}(M,\mathbb {Z} )}{\GV}_{g,\beta}{\frac {1}{k}}\left(2\sin \left({\frac {k\lambda }{2}}\right)\right)^{2g-2}q^{k\beta } .
\end{split}
\end{equation}
Therefore, relating QK and GV is in some sense a question of relating QK and GW.
In this paper, we completely solve this problem and relate theses two sets of invariants by an (invertible) integral linear transformation. This in particular implies the ad hoc definition of GV in \eqref{e:0.1} gives integral invariants in genus zero, as the quantum $K$-invariants are intrinsically \emph{integral} by definition.

Perhaps it is worth pointing out that these relations do not render quantum $K$-invariants, or GW, DT, PT, obsolete. Quantum $K$-theory has \emph{more direct} and \emph{different}
connections with finite-difference integrable systems, representation theory, and theoretical physics. Its relation to $3$ dimensional topological field theory can be found in the pioneering works of N.~Nekrasov, H.~Jockers, P.~Mayr etc.. See \cite{Jockers_Mayr_2019, Jockers_Mayr_2020} and references therein. For its connection to representation theory can be found in \cite{Okounkov_lectures} and references therein. Quantum $K$-theory of flag varieties is intimately related to finite difference integrable systems and quantum groups \cite{Givental_Lee_2003}. It has also inspired much progress in geometric combinatorics through works like \cite{Buch_2011, Buch_2013, Buch_2020} of A.~Buch, P.~Chaput, C.~Li, L.~Mihalcea, N.~Perrin and many others. These works are developed from the perspectives unique to the quantum $K$-theory.

We hope to extend these results to higher genus in future works.

\subsection{Contents of the paper}
In Section 1, a conjecture by Jockers and Mayr \cite{Jockers_Mayr_2019} is stated as Theorem~\ref{conjecture_npt} in essentially identical form. (Our degree zero invariants are different from theirs.) This theorem gives a precise (finite) integral linear transformation from Gopakumar--Vafa invariants to an ``essential subset'' of genus zero quantum $K$-invariants. In Proposition~\ref{p:1.5} we show that other genus zero quantum $K$-invariants can be commputed from these essential quantum $K$-invariants. Hence, the Main Theorem is reduced to Theorem~\ref{conjecture_npt}.

Sections~\ref{s:2} and \ref{s:3} are mainly devoted to the proof of Theorem~\ref{conjecture_npt}. Our approach is to show that the right hand side of the formula $\tilde{J}$ in Theorem~\ref{conjecture_npt} belongs to a special class of (generating) functions to which $J^K$ belongs. In the language of A.~Givental, this is called ``lying on the overruled Lagrangian cone." In Givental's framework, two functions both lying on the Lagrangian cone are related by a (generalized) mirror transformation, also known as Birkhoff factorization. 

In the case of quantum $K$-theory, Givental developed a characterization of all functions lying on the overruled Lagrangian cone, called the \emph{adelic characterization}.
In Section~\ref{s:2}, some necessary ingredients on A.~Givental's adelic characterization are briefly summarized. Further details can be found in Part I \cite{Chou_Lee_2022_QKGVI} or \cite{Givental_Tonita_2011}.

The adelic characterization is employed to show that $\tilde{J}$ lies on the Lagrangian cone in Section~\ref{s:3}. Furthermore, by the uniqueness theorem, the mirror transformation is shown to be trivial, proving Theorem~\ref{conjecture_npt}.

In the last section, we study the multiple cover formula on the local model. Assuming the "Virtual Clemen Conjecture" on Calabi-Yau threefold $X$, the Jockers--Mayr formula is derived from the ``first principles''. Of course, our proof as outlined above is completely different and admittedly much more involved, as the Virtual Clemens Conjecture remains a conjecture.

Indeed, this was how we first reached the formula \eqref{e:jtilde}, as at that time we were not able to understand the work of Jockers and Mayr \cite{Jockers_Mayr_2019}, written in a language we are not used to. Thanks to their explanations, we now understand that their formula, dating back to 2019, is essentially the same.

\subsection*{Acknowledgements}
We wish to thank A.~Givental, H.~Jokers, P.~Mayr, R.~Pandharipande, E.~Scheidegger, H.-H.~Tseng and Y.~Wen for their interest and discussions about this work. We are grateful to Jokers and Mayr for explaining their conjecture to us and verifying an earlier and more restrictive version of our formulation in \cite{Chou_Lee_2022_Bumsig} is consistent with their formulation.
This research is partially supported by the National Science and Technology Council, the Simons Foundation, Academia Sinica and University of Utah.

\section{GW, GV, QK and the JM Conjecture} \label{s:1}

\subsection{GW, GV, and QK}
Let $X$ be a smooth complex projective variety. 
(Cohomological) Gromov--Witten invariants of $X$ are defined to be
\[
\langle \tau_{d_1}(\phi_1)\dots \tau_{d_n}(\phi_n) \rangle^{X, {H}}_{g,n,\beta} := \int_{[\M_{g,n}(X,\beta)]^{\vir}}   \cup_{i=1}^n \ev_i^*(\phi_i) \ c_1(L_i)^{d_i} \in \mathbb{Q},
\]
where $[\M_{g,n}(X,\beta)]^{\vir}$ are the virtual fundamental classes, $\phi_1,\dots,\phi_n \in H(X)$, $d_1,\dots,d_n \in \mathbb{Z}_{\geq 0}$, $\ev_i$ and $L_i$ are the evaluation map and cotangent line bundle at the $i$-th marked point respectively. 

When $X$ is a Calabi-Yau threefold, $\M_{g,0}(X,\beta)$ has virtual dimension 0. We define 
\[
 {\GW}_{g, \beta} := \int_{[\M_{g,0}(X,\beta)]^{\vir} } 1 .
\]
By string, dilaton, and divisor equation, all Gromov--Witten invariant can be recovered from 0-pointed one.

In theoretical physics,     R.~Gopakumar and C.~Vafa in \cite{Gopakumar_Vafa_1998} introduced new \emph{integral} topological invariants on \emph{Calabi--Yau threefolds} (CY3) $X$, called \emph{Gopakumar--Vafa invariants}. These invariants represent the counts of ``numbers of BPS states'' on $X$. There have been various attempts at defining Gopakumar--Vafa invariants mathematically. We refer the readers to \cite{Maulik_Toda_2018} and references therein. 

A remarkable relation conjectured in \cite{Gopakumar_Vafa_1998} between GV and GW can be expressed in terms of generating functions as in \eqref{e:0.1}.
In genus 0, the Gopakumar--Vafa relation can be written as
\begin{equation} \label{e:1.2}
\begin{split}
    {\GW}_{g=0, \beta} =: {\GW}_{\beta} & = \sum_{k=1}^{\infty} \frac{1}{k^3} {\GV}_{\beta/k}, \\ 
    {\GV}_{g=0, \beta} =:     {\GV}_{\beta} & := \sum_{k=1}^{\infty} \frac{\mu(k)}{k^3} {\GW}_{\beta/k}, 
\end{split}
\end{equation}
where by definition
\[
\GW_{\beta/k} = \GV_{\beta/k}=0, \quad \text{if } \beta/k \notin H_2(X,\mathbb{Z}).
\]
We have used the M\"obius inversion and $\mu(e)$ is the M\"obius function. For our purpose, we will \emph{use \eqref{e:1.2} as the definition of ${\GV}_{\beta}$}.

\emph{Quantum $K$-invariants} are also \emph{integral} invariants \cite{Givental_2000, Lee_2004}.
For any smooth projective variety $X$, quantum $K$-invariants are defined as
\[
\langle \tau_{d_1}(\Phi_1)\dots \tau_{d_n}(\Phi_n) \rangle^{X, {K}}_{g,n,\beta} :=  \chi \Big(  \M_{g,n}(X,\beta) ; \Big(  \otimes_{i=1}^n \ev_i^*(\Phi_i) L_i^{d_i}  \Big) \otimes \O^{\vir}  \Big) \in \mathbb{Z}.
\]
Here $\Phi_1,\dots,\Phi_n \in K^0(X)$, $d_1,\dots,d_n \in \mathbb{Z}$, and $\O^{\vir}$ is the virtual structure sheaf on $\M_{g,n}(X,\beta)$. 

\subsection{Jockers--Mayr Conjecture} \label{section_conjecture}
H.~Jockers and P.~Mayr in \cite{Jockers_Mayr_2019} wrote down an explicit formula relating an ``essential subset'' of genus zero quantum $K$-invariants and Gopakumar--Vafa invariants. We state a slightly revised version of their formula in Theorem~\ref{conjecture_npt}.

For simplicity, $K(X)$ stands for the toplogical $K$-theory $K^0(X)$ of complex vector bundles.\footnote{With suitable modifications, the results apply to the Grothendieck group of algebraic vector bundles.} A basis of $K(X)$ is chosen as follows.
\[
\{
\Phi_{\alpha}
\}_{\alpha=1}^N =  \bigsqcup_{i=0}^3 \{ \Phi_{ij}\}_{j=1}^{n_i},
\]
where $\{ \ch ( \Phi_{ij} )\}_{j=1}^{n_i}$ forms a basis in $H^{2i}(X)$. In particular, $\{ \Phi_{0j} \}_{j=1}^{n_0} = \{ \Phi_{01} =\mathcal{O} \}$. Let $\{ \Phi^{ij} \}$ be the dual basis of $\{ \Phi_{ij} \}$ with respect to the $K$-theoretic Poincar\'e pairing
\[
(\Phi_a, \Phi_b)^{K} : = \chi (X, \Phi_a \Phi_b) = \int_X \td(TX)\ch(\Phi_a)\ch(\Phi_b).
\]
The following fact will be used later.
\begin{lemma} \label{lemma_degPhi}
For $i = 0$ and $1$,
\[ 
\ch( \Phi^{ij} ) \in H^{2(3-i)}(X).
\]
For $i=2$ and $3$,
\[
\ch ( \Phi^{ij}) \in H^{ \geq 2(3-i)} (X).
\]
\end{lemma}
\begin{proof}
$\Phi^{ij}$ can be written as:
\[
\Phi^{ij}:= \ch^{-1} \Big(  \td(TX)^{-1} {\rm PD} (\ch (\Phi_{ij})) \Big), 
\]
where PD denotes the Poincar\'e dual. The lemma follows from the definition of $\Phi_{ij}$ and that $\td(TX)^{-1} \in 1 + H^{\geq 4}(X)$.
\end{proof}


Let us fix some notation. Let 
\[
 t := \sum_{j=1}^{n_1} t_{j} \Phi_{1j}, \quad  \beta_j := \int_{\beta} \ch(\Phi_{1j}), \quad  \beta_t := \int_{\beta} \ch t := \sum_{j=1}^{n_1} t_{j} \beta_j.
\]
Define a $K$-theoretic $J$-function with input $t$
\[
J^K(t,q,Q)  := (1-q) + t + \sum_{\beta \in H_2(X,\mathbb{Z})}\sum_{ \alpha,n} \frac{Q^{\beta}}{n!} \Phi^{\alpha} \langle \frac{\Phi_{\alpha}}{1-qL}, t,\dots, t \rangle^{X,K}_{0,n+1,\beta}.
\]
This is the ``big $J$-function'' in Definition~\ref{def:KbigJ} with the input specialized to $\mathbf{t} =t$. 
Define
\begin{equation} \label{e:jtilde}
\begin{split}
    \Tilde{J}  :=   &(1-q) +
     t + \frac{ \sum_{j=1}^{n_1}\Phi^{1j} \partial_{t_j}F(t)  }{1-q} + \frac{q\ \Phi^{01} F(t)}{(1-q)^2} + \\
     &(1-q)\cdot\sum_{0\neq \beta \in H_2(X,\mathbb{Z}) } \sum_{r=1}^{\infty} \GV_{\beta}\, Q^{r\beta} \cdot
    \\
    \hspace{0mm} \cdot & \Bigg[ \sum_{j=1}^{n_1} \Phi^{1j} \beta_j\bigg(   a(r, q^r) + c(r,q,\beta_t) \bigg)  + \Phi^{01} \bigg( b(r,q^r) + d(r,q,\beta_t)  \bigg) \Bigg],
\end{split}
\end{equation}
where
\[
\begin{split}
    F(t) := &\frac{1}{3!}\langle t,t,t\rangle^{X,K}_{0,3,0} = \frac{1}{6}\sum_{i,j,k=1}^{n_1} t_it_jt_k\int_X \ch(\Phi_{1i})\ch(\Phi_{1j})\ch(\Phi_{1k}),
    \\
      a(r,q^r) &= \frac{(r-1)}{1-q^r} + \frac{1}{(1-q^r)^2},
    \\
     b(r,q^r) &= \frac{r^2-1}{1-q^r} + \frac{3}{(1-q^r)^2} -\frac{2}{(1-q^r)^3} ,
    \\
    c(r,q,\beta_t) &= \frac{\beta_t}{r(1-q)(1-q^r)} + \frac{ e^{\beta_t} -1 - \beta_t }{r^2(1-q)^2},
    \\
    d(r,q,\beta_t) &= \frac{\beta_t}{r(1-q)} \frac{ r(1-q^r) - q^r }{(1-q^r)^2} + \frac{ \beta_t^2 }{2r^2 (1-q)^2}
    \\
    & +\frac{ (1-3q)(e^{\beta_t} - 1 - \beta_t - \frac{\beta_t^2}{2}) + q \beta_t (e^{\beta_t} -1 -\beta_t) }{r^3(1-q)^3}.
\end{split}
\]
\begin{theorem} [{cf.~\cite{Jockers_Mayr_2019}}] \label{conjecture_npt}  
$J^K (t) = \tilde{J}$ for all Calabi--Yau threefolds.
\end{theorem}
The proof of this theorem will occupy the entire Section~\ref{s:3}.

\begin{remark}
    \textbf{(1)} $\beta=0$ term in the definition of $J^K(t,q,Q)$ gives nontrivial contribution:
    \[
        \frac{ \sum_{j=1}^{n_1}\Phi^{1j} \partial_{t_j}F(t)  }{1-q} + \frac{q\ \Phi^{01} F(t)}{(1-q)^2},
    \]
    which comes from the Poincar\'e pairing on $X$. We point out that our $\beta=0$ terms differ from those in \cite{Jockers_Mayr_2019}.

    \textbf{(2)} When $t=0$, we have $c(r,q,\beta_t) = d(r,q,\beta_t)=0$. In this case, Theorem~(\ref{conjecture_npt}) specializes to the form of Conjecture~1.1 in \cite{Chou_Lee_2022_QKGVI}.

    \textbf{(3)} One observes that $\Tilde{J}$ is a rational function in $q$, with poles only at roots of unity, and the orders of poles are no greater than $3$. This was explained in \cite[Corollary~3.3]{Chou_Lee_2022_QKGVI}.

    \textbf{(4)} The dependence of $\tilde{J}$ on $\beta_t$ is of the form $e^{\beta_t}$, except finitely many monomials in $\beta_t^{\leq 2}$ and an overall factor $\beta_t$ to some terms. This will be explained in geometric germs in Remark~\ref{r:3.14}.
\end{remark}

\subsection{From JM Conjecture to Main Theorem}
In this subsection we show that one can obtain all genus zero quantum $K$-invariants with cotangent lines $L_i$ (descendants) as finite linear integral combinations of Gopakumar--Vafa invariants, assuming Theorem~\ref{conjecture_npt}.

\begin{proposition} \label{p:1.5}
Theorem~\ref{conjecture_npt} implies Main Theorem.
\end{proposition}

\begin{proof}
We note that the knowledge of $J^K (t)$ in Theorem~\ref{conjecture_npt} is equivalent to that of a collection of genus zero quantum $K$-invariants with arbitrary insertions, including any powers of cotangent line bundles $L_1^{\otimes k}$, $k\in \mathbb{Z}$, at the first marked point, and only $\op{ch}^{-1} (H^2(X))$ insertions without $L_i$ at all remaining marked points. 
In order to get all genus zero quantum $K$-invariants, we will need to allow cotangent line bundles $L_i$'s and general $K$-classes at all marked points. This can be done as follows.

\textbf{1}. Inserting $L_i$ at the $i$-th point: From \cite[(2)]{Lee_Pandharipande_2004}, we have
\[
 L_i = L_1^{-1} \otimes \so_{D_{i | 1}},
\]
where $D_{i|1}$ is the virtual divisor whose general element is a map from a curve with a node separating the $1$-st and the $i$-th marked point to $X$. Inserting $L_i$ can be \emph{inductively} reduced to inserting additional $L_1$, which is already included in the formulation of Theorem~\ref{conjecture_npt}.

\textbf{2}. Inserting $K$-classes in $\op{ch}^{-1} (H^{\ge4} (X))$: By \eqref{e:3.2} in Lemma~\ref{lemma_vanishingdeg2} quantum $K$-invariants with insertions in $\op{ch}^{-1} (H^{\ge4} (X))$ vanish. Hence they can be omitted.

\textbf{3}. Inserting $K$-classes in $\op{ch}^{-1} (H^0(X))$: For $\op{ch}^{-1} (H^0(X))$, that is, $\so_X$, the $K$-theoretic \emph{string equation} in \cite[(22)]{Lee_2004} reduce it to the ``essential'' invariants already covered in Theorem~\ref{conjecture_npt}.

Therefore, all quantum $K$-invariants can be reduced to the essential invariants by simple explicit formulas.
\end{proof}

\section{Givental's adelic characterization} \label{s:2}
In this section, we briefly summarize results about adelic characterization in quantum $K$-theory in \cite{Givental_Tonita_2011} by Givental and Tonita. Further details can be found in Part I \cite[\S~2.3]{Chou_Lee_2022_QKGVI} or the original papers by Givental \cite{Givental_ER_Coh_K_2015, Givental_PEV_Toric_2015, Givental_PEVIII_ER}.

Let $\Lambda = \mathbb{Q}[\![Q]\!]$ be the Novikov ring and 
\[
 \mathbf{K} := K^0(X)\otimes \Lambda.
\]
Givental's loop space for quantum $K$-theory is defined as
\[
\K := K^0(X)(q) \otimes \Lambda.
\]
$\K$ has a natural symplectic structure with the symplectic form $\Omega$,
\[
\K \ni f,g \mapsto \Omega(f,g) := \Big( \Res_{q=0} + \Res_{q=\infty} \Big) ( f(q), g(q^{-1}) )^{K} \frac{dq}{q}.
\]
Here $( \cdot, \cdot )^{K}$ denotes the $K$-theoretic intersection pairing on $\mathbf{K}$:
\[
(a,b)^{K} := \chi(X, a \otimes b) = \int_X \td (T_X) \ch(a) \ch(b).
\]
$\K$ admits the following Lagrangian polarization with respect to $\Omega$:
\[
\begin{split}
    \K &= \K_+ \oplus \K_-
    \\
    & := \mathbf{K}[q,q^{-1}] \, \oplus \, \{ f(q) \in \K | f(0) \neq \infty ,\ f(\infty) =0 \}.
\end{split}
\]
\begin{definition}\label{def:KbigJ}
The \emph{big $J$-function} of $X$ in the quantum $K$-theory is defined as a map $\K_+ \rightarrow \K$:
\[
   \mathbf{t} \mapsto J^{\KK}(\mathbf{t}) := (1-q) + \mathbf{t}(q) + \sum_{\alpha} \Phi^{\alpha}\sum_{n,\beta} \frac{Q^{\beta}}{n!} \langle \frac{\Phi_{\alpha}}{1-qL}, \mathbf{t}(L),\dots, \mathbf{t}(L) \rangle ^{X,K}_{0,n+1,\beta},
\]
where $\{ \Phi_{\alpha} \}$ and $\{\Phi^{\alpha}\}$ are Poincar\'e-dual basis of $K^0(X)$ with respect to $( \cdot, \cdot )^{K}$. 
The  \emph{overruled Lagrangian cone} $\mathcal{L}^K \subset \mathcal{K}$ is defined to be the range $J^K (\K_+)$. 
\end{definition}

Applying the virtual Kawasaki-Hirzebruch--Riemann--Roch formula for Deligne--Mumford stacks (VKHRR) \cite{Tonita_VKHRR} on $\M_{g,n}(X,\beta)$, we get  
\[  \begin{split}
J^{\KK}(\mathbf{t}) &=(1-q) + \mathbf{t}(q) + \sum_{n,\beta,\alpha} \frac{Q^{\beta} \Phi^{\alpha}}{n!}\langle \frac{\Phi_{\alpha}}{1-qL}, \mathbf{t}(L),\dots, \mathbf{t}(L)  \rangle^{X,K}_{0,n+1,\beta}
\\
&= (1-q) + \mathbf{t}(q) + \sum_{n,\beta,\alpha} \frac{Q^{\beta} \Phi^{\alpha}}{n!}\sum_{\zeta} \langle \frac{\Phi_{\alpha}}{1-qL}, \mathbf{t}(L),\dots,\mathbf{t}(L)  \rangle^{X_{\zeta}}_{0,n+1,\beta},
\end{split}\]
where in $\sum_{\zeta}$ $\zeta$ runs through all roots of unity (including 1), $X_{\zeta}$ stand for the collection of inertia stacks (``Kawasaki strata'') where $g$ acts on $L_1$ with eigenvalue $\zeta$, and $\langle ... \rangle^{X_{\zeta}}$ denote the contributions of $X_{\zeta}$ in the Riemann--Roch formula. In other words, $\langle ...\rangle^{X,K}$ represent (true) quantum $K$-invariants while $\langle ... \rangle^{X_{\zeta}}$ stand for (collections of) \emph{cohomological} invariants. Symbolically, we have 
\[
\langle \cdots \rangle^{X,K} = \sum_{\zeta} \langle \cdots\rangle^{X_{\zeta}}.
\]

Let $\zeta$ be a primitive $r$-th roots of unity. Recall in Part I (taken from \cite{Givental_Tonita_2011}) that we use arm$(L)$, leg$_{r}(L)$, and tail$_{\zeta}(L)$ to denote certain specific contributions from Kawasaki strata in the Riemann--Roch formula for stacks:
\begin{equation} \label{eqn_arm_leg_tail}
\begin{split}
    {\rm arm}(q) &= \sum_{n,\beta \neq 0,\alpha} \frac{Q^{\beta} \Phi^{\alpha}}{n!}\sum_{\zeta' \neq 1} \langle \frac{\Phi_{\alpha}}{1-qL}, \mathbf{t}(L),\dots,\mathbf{t}(L)  \rangle^{X_{\zeta'}}_{0,n+1,\beta} ;
    \\
    {\rm leg}_{r}(q) &= \Psi^r ({\rm arm}(q)|_{\mathbf{t}=0} );
    \\
    {\rm tail}_{\zeta}(q) &= \sum_{n,\beta \neq 0,\alpha} \frac{Q^{\beta} \Phi^{\alpha}}{n!}\sum_{\zeta' \neq \zeta} \langle \frac{\Phi_{\alpha}}{1-qL}, \mathbf{t}(L),\dots,\mathbf{t}(L)  \rangle^{X_{\zeta'}}_{0,n+1,\beta} ,
\end{split}
\end{equation}
where $\zeta'$ in the above sums are arbitrary roots of unity and
$\Psi^r$ are the \emph{Adams operations}. Recall that Adams operations are additive and multiplicative endomorphisms of $K$-theory or more generally $\lambda$-rings, acting on line bundles by $\Psi^r(L) = L^{\otimes r}$. Here $\Psi^r$ also act on the Novikov variables by $\Psi^r(Q^{\beta}) = Q^{r\beta}$ and on $q$ by $\Psi^r(q) = q^r$.


Another necessary ingredient in the proof is the \emph{fake quantum $K$-theory},  defined by ``naively'' applying the virtual Riemann--Roch for schemes to stacks
\[
\langle \tau_{d_1}(\Phi_1)\dots \tau_{d_n}(\Phi_n) \rangle^{X, \fake}_{g,n,\beta} := \int_{[\M_{g,n}(X,\beta)]^{\vir} } \td(T_{\M_{g,n}(X,\beta)}^{\vir}) \ch ( \otimes_{i=1}^n \ev_i^*(\Phi_i) L_i^{d_i} ),
\]
where $[\M_{g,n}(X,\beta)]^{\vir}$ is the (cohomological) virtual fundamental class and $[T^{\vir}]$ the virtual tangent bundle. 
There is a simlar Givental formalism for the fake quantum $K$-theory. Givental's loop space is changed to formal series
\[
\K^{\fake} := \mathbf{K}[\![ 1-q, (1-q)^{-1}]
\]
%
with the symplectic form $\Omega^{\fake}$
\[
\Omega^{\fake}(f,g) := - \Res_{q=1} (f(q), g(q^{-1}))^{K} \frac{dq}{q},
\]
and a corresponding Lagrangian polarization:
\[
\begin{split}
    \K^{\fake} &= \K_+^{\fake} \oplus \K_-^{\fake}
    \\
 : &= \mathbf{K}[\![q-1]\!] \oplus {\rm span}_{\mathbf{K}} \left\{ \frac{q^k}{(1-q)^{k+1}}  \right\}_{k \geq 0}.
\end{split}
\]
The big $J$-function of $X$ in the fake $K$-theory is defined as a map $\K_+^{\fake} \rightarrow \K^{\fake}$:
\[
\mathbf{t} \mapsto J^{\fake} (\mathbf{t}) := (1-q) + \mathbf{t}(q) + \sum_{\alpha} \Phi^{\alpha} \sum_{n,d} \frac{Q^d}{n!} 
\langle \frac{\Phi_{\alpha}}{1-qL}, \mathbf{t}(L),\dots,\mathbf{t}(L) \rangle ^{X,\fake}_{0,n+1,d}.
\]

The definitions in \eqref{eqn_arm_leg_tail} can be interpreted as formal power series as in Part I. For this purpose a more general version of arm and tail, which allows ``non-geometric'' input $f\in \mathcal{K}$, will be defined. 
Let $f\in \mathcal{K}$ and $\zeta$ be any primitive $r$-th roots of unity with $r\geq 2$. Define
\begin{equation} \label{e:2.2}
\begin{split}
    \arm (f;q) &:= \pi^{\fake}_+ \Big( f|_{q=1} \Big) - \pi_+(f)|_{q=1}  \in \mathbf{K}[\![ 1-q ]\!];
    \\
    \leg_{r}(q) &:= \Psi^r \Big( \pi_+^{\fake} \Big(J^K(0)|_{q=1} \Big) \Big) \in \mathbf{K}[\![ 1-q ]\!];
    \\
    \tail_{\zeta} (f;q) & := \pi_+^{1-\zeta q}\Big( f|_{q=\zeta^{-1}}  \Big) - \pi_+(f)|_{q= \zeta^{-1}}  \in \mathbf{K}[\![ 1-\zeta q ]\!];
    \\
 \delta_{\zeta} (f; q) & := \Big[ 1- \zeta^{-1} q +  \pi_+(f)(\zeta^{-1} q) + {\rm tail}_{\zeta}(f; \zeta^{-1} q) \Big]_{q=1} \\
  & \quad \in \mathbf{K}[\![ 1-q ]\!],
\end{split}
\end{equation}
where
\[
\begin{split}
    &\pi_+ : \K \rightarrow \K_+ ;
    \\
    &\pi_+^{\fake} : \K^{\fake} \rightarrow \K_+^{\fake} ;
    \\
    &\pi_+^{1-\zeta q} : \mathbf{K}[\![ 1-\zeta q, (1-\zeta q)^{-1}] \rightarrow \mathbf{K}[\![ 1-\zeta q ]\!]
\end{split}
\]
are projections and \emph{$[...]|_{q=\zeta}$ stands for series expansion at $q=\zeta$}.

\begin{remark} \label{r:2.2}
Suitably interpreted, the generalized definitions above are actually the same as those in \eqref{eqn_arm_leg_tail} when $f= J^K(\mathbf{t})$. For example, we will see in Proposition~\ref{prop_head_fake} that
\[
\begin{split}
    &\pi_+^{\fake} \Big(J^K(\mathbf{t})|_{q=1} \Big)  - \pi_+(J^K(\mathbf{t}))|_{q=1}
    \\
    & = \pi_+^{\fake} \Big(J^K(\mathbf{t})|_{q=1} \Big) - \Big[ (1-q) + \mathbf{t} \Big] =  {\rm arm}(q),
\end{split}
\]
and hence two definitions of $\mathrm{arm}(q)$ agree.
The $\leg_r(q)$ in \eqref{e:2.2} is consistent with $\leg_r(q)|_{q=1}$ in equation~(\ref{eqn_arm_leg_tail}). 
With this understanding, when $f$ takes the form
\[
    1 - q + \mathbf{t} + f_-( \mathbf{t}, q ) 
\]
with $f_-(\mathbf{t},q) \in \K_-$, we \emph{abuse the notation} and do not distinguish these two.
\end{remark}

\begin{proposition}[Essentially {\cite[Proposition~1]{Givental_Tonita_2011}}]\label{prop_head_fake}
\[
J^{\KK}(\mathbf{t})|_{q=1} = J^{\fake}(\mathbf{t} + {\rm arm}).
\]
\end{proposition}
\quad

For $\zeta\neq 1$, a primitive $r$-th roots of unity, another ``Kawasaki strata'' will also be used, namely, the \emph{stem space}, which is isomorphic to
\[
 \M_{0,n+2,\beta}^X(\zeta) := \M_{0,n+2} \left( [X/\mathbb{Z}_r], \beta; (g,1,\dots,1,g^{-1}) \right) .   
\]
Here the group elements $g, 1, g^{-1}$ signal the twisted sectors in which the marked points lie. As in Part I, the notation $[...]^{X_{\zeta}}$ will be reserved for stem contributions
\begin{equation} \label{e:stem}
\begin{split}
    &\Big[ T_1(L), T(L),\dots, T(L), T_{n+2}(L) \Big]^{X_{\zeta}}_{0,n+2,\beta}
    \\
 := &\int_{[\M_{0,n+2,\beta}^X(\zeta)]^{\vir}} \td( T_{\M} ) \ch \left( 
    \frac{ \ev_1^* (T_1(L)) \ev_{n+2}^* (T_{n+2}(L)) \prod_{i=2}^{n+1} \ev_i^* T(L) }
    {\Tr (\Lambda^* N^*_{\M})}  \right),
\end{split}
\end{equation}
where $[\M_{0,n+2,\beta}^X(\zeta)]^{\vir}$ is the virtual fundamental class, $T_{\M}$ is the (virtual) tangent bundle to $\M_{0,n+2,\beta}^X(\zeta)$, and $N_{\M}$ is the (virtual) normal bundle of $\M_{0,n+2,\beta}^X(\zeta) \subset \M_{0,nr+2}(X,r\beta)$. We denote the trace bundle, $\Tr(F)$, the virtual bundle
\[
\Tr (F) := \sum_{\lambda} \lambda F_{\lambda},
\]
where the sum is over the eigenvalue of $g$-action with the corresponding eigenspace. 

\begin{proposition} [{ \cite[\S~8]{Givental_Tonita_2011} }] \label{Prop_Stem_inv}
Let $\zeta$ be a primitive $r$-th roots of unity.
We have
\[
\begin{split}
    & \sum_{n,\beta,\alpha}  \frac{Q^{\beta} \Phi^{\alpha}}{n!}\langle \frac{\Phi_{\alpha}}{1-qL}, \mathbf{t}(L),\dots,\mathbf{t}(L)  \rangle^{X_{\zeta}}_{0,n+1,\beta} 
    \\
   = & \sum_{n,\beta,\alpha} \frac{Q^{r\beta}\Phi^{\alpha}}{n!} \Big[ \frac{\Phi_{\alpha}}{1-q\zeta L^{1/r}}, {\rm leg}_{r}(L),\dots, {\rm leg}_{r}(L),  \delta_{\zeta}(L^{1/r}) \Big]^{X_{\zeta}}_{0,n+2,\beta},
\end{split}
\]
where $\delta_{\zeta}(q) =\delta_{\zeta}(J^K; q)$ in \eqref{e:2.2}.
\end{proposition}

\begin{remark} \label{r:2.5}
Since $L^{1/r}$ is unipotent on a fixed moduli, the following expansion
\[
\frac{1}{1-q\zeta L^{1/r}} = \sum_{i \geq 0} \frac{ (\zeta q)^i (L^{1/r} -1)^i }{ (1-\zeta q)^{i+1} }
\]
is a finite sum. When combined with the VKHRR for $\M_{0,n}(X,\beta)$, this expresses $J^{\KK}(\mathbf{t},q) -(1-q) - \mathbf{t}(q)$ as a sum of rational functions with poles only at roots of unity.
\end{remark}

\begin{theorem}[Adelic characterization, {\cite[\S~6]{Givental_Tonita_2011} }] \label{thm_adelic_charac}
$f\in \mathcal{L}^K$ if and only if
\begin{itemize}
    \item[(i)]$f|_{\mathcal{K}_-}$ has poles only at roots of unity.
    \item[(ii)] $f|_{q=1} = J^{\fake} \left( -(1-q) + \pi_+(f)|_{q=1} + \arm(f) \right)$.
    \item[(iii)] Let $\zeta$ be any primitive $r$-th roots of unity with $r\geq 2$. We have
    \[
    \begin{split}
        & f|_{q=\zeta^{-1}} = \delta_{\zeta}(f; \zeta q)
        \\
        &+ \sum_{n,\beta,\alpha} \frac{Q^{r\beta} \Phi^{\alpha}}{n!} \Big[ \frac{\Phi_{\alpha}}{1-\zeta q L^{1/r}}, \leg_{r}(L) ,\dots, \leg_{r}(L), \delta_{\zeta}(f;L^{1/r}) \Big]^{{X}_{\zeta}}_{0,n+2,\beta}.
    \end{split}
    \]
\end{itemize}
\end{theorem}
\begin{remark}
    Our formulation of criteria (ii) and (iii) in adelic characterization might appear different from the ones in \cite[\S~6]{Givental_Tonita_2011}. They are actually equivalent.

    The criterion (ii), $f|_{q=1} \in \mathcal{L}^{\fake}$, in \cite[\S~6]{Givental_Tonita_2011} is equivalent to our formulation since if $f\in \mathcal{
    L}^K$, then 
    \[
    f|_{q=1} = J^{\fake}( -(1-q) + \pi_+^{\fake}(f) ) = J^{\fake}( -(1-q) + \pi_+(f) + \arm(f) ).
    \]
    Conversely, $J^{\fake}(-(1-q) + \pi_+(f) + \arm(f)) \in \mathcal{L}^{\fake}$ is clear.
    
    The criterion (iii) in \cite[\S~6]{Givental_Tonita_2011} is formulated as $f|_{q=\zeta^{-1}} \in T^r$. Here $T^r$ is some tangent space of $\mathcal{L}^{\fake}$ twisted by some characteristic classes and with suitable change of variable. This criterion is exactly the reformulation of the stem theory computation, see the Corollary in \cite[\S~8]{Givental_Tonita_2011} for detail. Our formalism is exactly the corresponding stem theory equality.
\end{remark}

\section{Proof of $\tilde{J} =J^K$}  \label{s:3} 

In this section, we use adelic characterization (Theorem~\ref{thm_adelic_charac}) to show that $\tilde{J}$ defined in \eqref{e:jtilde} $\in \mathcal{L}^K$. It follows from a uniqueness theorem that the mirror transformation is trivial and $\tilde{J} = J^K$. Combined with Proposition~\ref{p:1.5} this concludes the proof of the Main Theorem.

\subsection{Outline of the proof}
Here we provide an outline and sign posts of the proof.

\textbf{(1)}. In Section~\ref{s:3.2} we recall basic facts about (twisted) Gromov--Witten invariants on Calabi-Yau threefold. In particular, a vanishing theorem which was used in Proposition~\ref{p:1.5} is proved.

\textbf{(2)}. To check the adelic characterization,  $\tilde{J}$ is expanded in partial fractions with poles at $(1-\zeta q)$ for all roots of unity $\zeta$. Note that the orders of poles are $\leq 3$ by definition of $\Tilde{J}$.
The computation is divided into four parts in Lemmas~\ref{lemma_a_expansion}, \ref{lemma_b_expansion}, \ref{lemma_c_expansion}, and \ref{lemma_d_expansion} for $a(r,q^r)$, $b(r,q^r)$, $c(r,q,\beta_t)$, and $d(r,q,\beta_t)$ terms respectively.

\textbf{(3)}. Motivated by Proposition~\ref{prop_head_fake} and Proposition~\ref{Prop_Stem_inv}, we emulate the geometric situation to formulate the following functions for the $\tilde{J}$ 
\[
\begin{split}
    \widetilde{\arm} (q) &:= \arm( \tilde{J}; q ) \in \mathbf{K}[\![ 1-q ]\!],
    \\
    \widetilde{\leg_r}(q) & := \Psi^r \Big( \widetilde{\arm} (q)|_{t=0} \Big) \in \mathbf{K}[\![ 1-q ]\!],
    \\
    \widetilde{\tail}_{\zeta} (q) & := \tail_{\zeta}(\tilde{J};q) \in \mathbf{K}[\![ 1-\zeta q ]\!], \\
    \tilde{\delta}_{\zeta}(q) &:= \delta_{\zeta} ( \tilde{J}; q ) \in \mathbf{K}[\![ 1- q ]\!],
\end{split}
\]
in Section~\ref{s:3.4}. They serve as bridges for comparison with the corresponding functions associated to $J^K$ from the Kawasaki strata.

\textbf{(4)}. In Proposition~\ref{prop_J_q=1} and Proposition~\ref{prop_Jzeta_tnot0}, we show that $\widetilde{\arm}$, $\widetilde{\leg}$ and $\widetilde{\tail}$ satisfy the properties of their geometric counterparts in Propositions~\ref{prop_head_fake} and \ref{Prop_Stem_inv}. 
We then apply the adelic characterization (Theorem~\ref{thm_adelic_charac}) to conclude $\tilde{J} \in \mathcal{L}^K$, and moreover $\tilde{J} = J^K(t)$.

\subsection{Twisted Gromov--Witten theory on CY3} \label{s:3.2}
Recall the notation in Part I of the \emph{stem invariants}, which are important organizing ingredients in the virtual Hirzebruch--Riemann--Roch for stacks \cite{Givental_Tonita_2011} \cite{Chou_Lee_2022_QKGVI} 
\[
\begin{split}
    & \Big[ T_1(L), T(L),\dots, T(L), T_{n+2}(L) \Big]^{X_{\zeta}}_{0,n+2,\beta}
    \\
    := & \int_{[\M_{0,n+2,\beta}^X(\zeta)]^{\vir}} \td( T_{\M} ) \ch \left( 
    \frac{ \ev_1^* (T_1(L)) \ev_{n+2}^* (T_{n+2}(L)) \prod_{i=2}^{n+1} \ev_i^* T(L) }
    {\Tr (\Lambda^* N^*_{\M})}  \right).
\end{split}
\]
This can be interpreted as \emph{twisted} (cohomological) GW invariants with the twisting class
\[
\td (T_{\M}) \ch \Big( \frac{1}{\Tr (\Lambda^* N^*_{\M})} \Big).
\]
Such twisting classes come from the deformation theory of the moduli of stable maps and consist of three parts. See e.g., \cite[Section~8]{Givental_Tonita_2011}.
\begin{enumerate}
    \item[type $A$.] $\displaystyle \td(\pi_*^K \ev^*(TX))\prod_{k=1}^{r-1} \td_{\zeta^k} (\pi_*^K \ev^*(T_X \otimes \mathbb{C}_{\zeta^k})) $, where $\pi: \mathcal{C} \rightarrow \M$ and $\ev: \mathcal{C} \rightarrow X/ \mathbb{Z}_r$ form the universal (orbifold) stable map diagram. $\mathbb{C}_{\zeta^k}$ is the line bundle over $B\mathbb{Z}_r$ with $g$ acts by $\zeta^k$, and for any line bundle $l$, define the invertible multiplicative characteristic classes
    \[
    \td(l) := \frac{c_1(l)}{1-e^{-c_1(l)}}, \quad \td_{\lambda}(l) := \frac{1}{1-\lambda e^{-c_1(l) }}.
    \]
    \item[type $B$.] $\displaystyle \td(\pi_*^K(-L^{-1})) \prod_{k=1}^{r-1} \td_{\zeta^k} (\pi_*^K(-L^{-1} \otimes \ev^* (\mathbb{C}_{\zeta^k}) ))$, where $L=L_{n+3}$ is the universal cotangent line bundle of 
    \[
    \mathcal{C} \cong \M_{0,n+3}^{X/\mathbb{Z}_r ,\beta} (g, 1,\dots,1,g^{-1},1).
    \]
    \item[type $C$.] $\displaystyle \td^{\vee}(-\pi_*^K i_* \mathcal{O}_{Z_g}) \td^{\vee}(-\pi_*^K i_*\mathcal{O}_{Z_1}) \prod_{i=1}^{k-1}\td^{\vee}_{\zeta^k}(-\pi_*^K i_*\mathcal{O}_{Z_1})$, where $Z_1$ stands for unramified nodal locus, and $Z_g$ stands for ramified one with $i: Z \rightarrow \mathcal{C}$ the embedding of nodal locus. For any line bundle $l$,
    \[
     \td^{\vee}(l) = \frac{-c_1(l)}{1-e^{c_1(l)}}, \quad \td^{\vee}_{\lambda}(l) = \frac{1}{1-\lambda e^{c_1(l) }}.
    \]
\end{enumerate}
We start with the following two observations.
\begin{lemma} \label{lemma_rpower}
\[
\td (T_{\M}) \ch \Big( \frac{1}{\Tr (\Lambda^* N^*_{\M})} \Big) =: r^{-n} + T_{0,n+2,\beta}(\zeta),
\]
where $T_{0,n+2,\beta}(\zeta)\in H^{>0}(\M^X_{0,n+2,\beta}(\zeta))$, with $0$, $n$, $\beta$, and $\zeta$ inherited from $\M = \M^X_{0,n+2,\beta}(\zeta)$.
\end{lemma}
\begin{proof}
Note that $N_{\M}$ has virtual dimension $(r-1)n$ since $\M$ is considered as Kawasaki strata in $\M^X_{0,nr+2,\beta}$. 

Only the twisting classes on normal bundle give nontrivial constant. It is of the following form:
\[
\prod_{k=1}^{r-1} \td_{\zeta^k} \Big( \pi^K_*(  E \otimes \ev^* (\mathbb{C}_{\zeta^k}) )\Big),
\]
where $E$ is a virtual bundle of rank $n$. The constant term is given by
\[
\Big( \prod_{k=1}^{r-1} \frac{1}{1-\zeta^k} \Big)^n = r^{-n}.
\]

\end{proof}
\begin{lemma} \label{lemma_vanishingdeg2}
Let $X$ be a Calabi-Yau threefolds $X$, if $\deg_{\mathbb{C}} \phi_1 \geq 2$ then
\begin{equation} \label{e:3.1}
\langle \tau_{k_1}(\phi_1),\dots, \tau_{k_n}(\phi_n) \rangle_{g,n,\beta\neq 0}^{\#} =0,
\end{equation}
where $\#$ denotes cohomological GW invariants with any twisting.

In $K$-theory let $\Phi \in K(X)$. If $\deg_{\mathbb{C}} \ch (\Phi_1) \geq 2$, we have
\begin{equation} \label{e:3.2}
    \langle \tau_{k_1}(\Phi_1),\dots, \tau_{k_n}(\Phi_n) \rangle_{g,n,\beta\neq 0}^K =0.
\end{equation}
\end{lemma}

\begin{proof}
Let us start with the \eqref{e:3.1}
\[
 \pi_1 : \M_{g,n}(X, \beta) \to \M_{g,1} (X, \beta)
\]
be the forgetful map forgetting the last $n-1$ marked points and $T\in H^*(\M_{g,n}(X,\beta))$ be the twisting class.
By projection formula
\[
\begin{split}
\int_{[\M_{g,n}(X, \beta)]^{\vir}} T \ &\prod_{i=1}^{n} \Big(\psi_i^{k_i} \ev_i^*\phi_i\Big) 
\\
&= \int_{[\M_{g,1}(X, \beta)]^{\vir}} (\ev_1^* \phi_1 ) \ (\pi_1)_*\left( T \ \psi_1^{k_1}\prod_{i=2}^{n} \psi_i^{k_i} \ev_i^*\phi_i \right) .
\end{split}
\] 
Since $[\M_{g,1}(X,\beta)]^{\vir}$ has virtual dimension 1, while $\deg_{\mathbb{C}}(\phi_i) \geq 2$, the integral in the last equation must vanish. This shows \eqref{e:3.1}.

Similar argument for the $K$-theory, let $F\in K(\overline{M}_{g,n}(X,\beta))$ be the twisting bundle. By projection formula
\[
\begin{split}
    \chi \Big( &\overline{M}_{g,n}(X,\beta) ; \mathcal{O}^{\vir}\otimes F \otimes\Big( \otimes_{i=1}^n L_i^{d_i}\ev_i^*(\Phi_i)  \Big)\Big)
    \\
    & = \chi \Big( \overline{M}_{g,1}(X,\beta); \mathcal{O}^{\vir} \otimes \ev_1^* \Phi_1 \otimes(\pi_1^K)_*\Big(F \otimes L_1^{d_1}\otimes_{i=2}^n L_i^{k_i} \ev_i^* \Phi_1 \Big) \Big)
    =0
\end{split}
\]
where the vanishing holds for the same dimensional reason. \eqref{e:3.2} is now proved.

\end{proof}

\subsection{Expansion of $\tilde{J}$} \label{s:3.3}
In this subsection, we rewrite $\tilde{J}$ in terms of the sum of a polynomial and fractions of the form $(1-\zeta q)^{-i}$ for $\zeta$ any roots of unity and $i \in \mathbb{Z}>0$. The computation will be long and tedious. We separate it into four parts. 

We first fix some notations/conventions. $\sum_{\zeta^r=1}$ is the sum over primitive $r$-th roots of unity. For $\beta\in H_2(X,\mathbb{Z}) \setminus \{0\}$, 
\[
    \ind(\beta) := \max \Big\{ k\in \mathbb{N} \Big| \beta/k \in H_2(X,\mathbb{Z}) \Big\}.
\]
Finally for $\ind (\beta)=1$, we define
\[
    \GV_{d\beta}^{(\gamma)} := \sum_{k|d} k^{\gamma} \GV_{k\beta}.
\]
\begin{lemma} \label{lemma_a_expansion}
\[
\begin{split}
    &(1-q)\sum_{j=1}^{n_1}\Phi^{1j}\sum_{0\neq \beta\in H_2(X,\mathbb{Z})} \sum_{r=1}^{\infty} \GV_{\beta} Q^{r\beta} \beta_j a(r,q^r)
    \\
    & =(1-q)\sum_{j=1}^{n_1}\Phi^{1j}\sum_{{\rm ind}(\beta)=1} \sum_{r=1}^{\infty} \sum_{\zeta^r=1} \sum_{d=1}^{\infty}Q^{rd\beta} \beta_j \Bigg( \frac{\GV_{d\beta}^{(1)} - \frac{1}{r^2d^2} \GV_{d\beta}^{(3)} }{1-\zeta q} + \frac{\frac{1}{r^2d^2}\GV_{d\beta}^{(3)} }{(1-\zeta q)^2} \Bigg)
    \\
    & = \sum_{j=1}^{n_1}\Phi^{1j}\sum_{{\rm ind}(\beta)=1} \sum_{r=1}^{\infty} \sum_{\zeta^r=1} \sum_{d=1}^{\infty}Q^{rd\beta} \beta_j \Bigg( \zeta^{-1} \Big( \GV_{d\beta}^{(1)} - \frac{1}{r^2d^2} \GV_{d\beta}^{(3)} \Big)
    \\
    & + \frac{(1-\zeta^{-1})(\GV_{d\beta}^{(1)} - \frac{1}{r^2d^2} \GV_{d\beta}^{(3)}) + \zeta^{-1}(\frac{1}{r^2d^2}\GV_{d\beta}^{(3)}) }{1-\zeta q} + \frac{ (1-\zeta^{-1})(\frac{1}{r^2d^2} \GV_{d\beta}^{(3)} ) }{(1-\zeta q)^2}\Bigg),
\end{split}
\]
\end{lemma}
\begin{proof}
For the first equality, we use the following expansions:
\[
\begin{split}
    \frac{1}{1-q^r} &= \sum_{k|r}\sum_{\zeta^k=1} \frac{1}{r(1-\zeta q)};
    \\
    \frac{1}{(1-q^r)^2} &= \sum_{k|r}\sum_{\zeta^k=1} \Big( \frac{1}{r^2(1-\zeta q)^2} + \frac{r-1}{r^2(1-\zeta q)} \Big).
\end{split}
\]
Let $\zeta$ be any primitive $r$-th roots of unity with $r|M$ and $\ind(\beta')=1$. We compute
\[
\begin{split}
    &\Coeff\Big(\sum_{0\neq \beta\in H_2(X,\mathbb{Z})} \sum_{r=1}^{\infty} \GV_{\beta} Q^{r\beta} a(r,q^r), \frac{Q^{M\beta'}}{1-\zeta q} \Big)
    \\
    & =\Coeff \Bigg(  \sum_{k| \frac{M}{r}}\Big( \frac{ \frac{M}{rk}(rk-1) }{1-q^{rk}} + \frac{ \frac{M}{rk} }{(1-q^{rk})^2} \Big) \GV_{M\beta' /rk} ,  \frac{1}{1-\zeta q}  \Bigg)
    \\
    & = \Coeff \Bigg( \sum_{k|\frac{M}{r}} \Big( \frac{\frac{M}{rk}(rk-1)}{rk(1-\zeta q)} + \frac{\frac{M}{rk}(rk-1)}{(rk)^2(1-\zeta q)} \Big)\GV_{M\beta'/rk}, \frac{1}{1-\zeta q}   \Bigg)
    \\
    & = \GV_{M\beta'/r}^{(1)} - \frac{1}{M^2} \GV_{M\beta'/r}^{(3)}.
\end{split}
\]
The computation for the coefficient of $\frac{1}{(1-\zeta q)^2}$ is similar. 

For the second equality, we write $1-q = (1-\zeta^{-1}) + \zeta^{-1}(1-\zeta q)$. This finish the computation.
\end{proof}
\begin{lemma} \label{lemma_b_expansion}
\[
\begin{split}
    &(1-q)\Phi^{01} \sum_{0\neq \beta \in H_2(X,\mathbb{Z})} \sum_{r=1}^{\infty} \GV_{\beta} Q^{r\beta} b(r,q^r)
    \\
    & =(1-q)\Phi^{01}\sum_{{\rm ind}(\beta)=1} \sum_{r=1}^{\infty} \sum_{\zeta^r=1} \sum_{d\geq 1} Q^{rd\beta} 
    \\
    &\hspace{20mm} \Bigg( \frac{rd \GV^{(-1)}_{d\beta} - \frac{1}{r^3d^3}\GV_{d\beta}^{(3)}}{1-\zeta q} 
     + \frac{ \frac{3\GV_{d\beta}^{(3)}}{r^3d^3} }{(1-\zeta q)^2} + \frac{ \frac{-2\GV_{d\beta}^{(3)}}{r^3d^3} }{(1-\zeta q)^3}  \Bigg)
    \\
    &= \Phi^{01}\sum_{{\rm ind}(\beta)=1} \sum_{r=1}^{\infty} \sum_{\zeta^r=1} \sum_{d\geq 1} Q^{rd\beta} \Bigg( \zeta^{-1} \Big( rd \GV^{(-1)}_{d\beta} - \frac{1}{r^3d^3}\GV_{d\beta}^{(3)} \Big)
    \\
    & \hspace{20mm} + \frac{ (1-\zeta^{-1}) (rd \GV^{(-1)}_{d\beta} - \frac{1}{r^3d^3}\GV_{d\beta}^{(3)}) + \zeta^{-1}( \frac{3\GV_{d\beta}^{(3)}}{r^3d^3} ) }{1-\zeta q} 
    \\
    & \hspace{20mm} + \frac{(1-\zeta^{-1})(\frac{3\GV_{d\beta}^{(3)}}{r^3d^3}) + \zeta^{-1}( \frac{-2\GV_{d\beta}^{(3)}}{r^3d^3} ) }{(1-\zeta q)^2} + \frac{(1-\zeta^{-1})(\frac{-2\GV_{d\beta}^{(3)}}{r^3d^3} )}{(1-\zeta q)^3}\Bigg).
\end{split}
\]
\end{lemma}
\begin{proof}
    The computation is similar to Lemma~\ref{lemma_b_expansion} with the only difference that one more expansion is used:
    \[
    \frac{1}{(1-q^r)^3} = \sum_{k|\frac{M}{r}}\sum_{\zeta^k=1} \Big( \frac{1}{r^3(1-\zeta q)^3} + \frac{3(r-1)}{2r^3(1-\zeta q)^2} + \frac{2r^2-3r+1}{2r^3(1-\zeta q)}   \Big).
    \]
    We skip the computation.
\end{proof}
Expansions for $c(r,q,\beta_t)$ and $d(r,q,\beta_t)$ can be computed using the same method.
\begin{lemma} \label{lemma_c_expansion}
\[
\begin{split}
    &(1-q)\sum_{j=1}^{n_1}\Phi^{1j}\sum_{0\neq \beta \in H_2(X,\mathbb{Z})} \sum_{r=1}^{\infty} \GV_{\beta} Q^{r\beta} \beta_j c(r,q,\beta_t)
    \\
    & = \sum_{j=1}^{n_1}\Phi^{1j} \sum_{{\rm ind}(\beta)=1} \sum_{d=1}^{\infty} \beta_j\beta_t 
    \\
    & \hspace{10mm} \cdot\Bigg[ \sum_{r\geq 2}\sum_{\zeta^r=1} Q^{rd\beta} \Bigg( \frac{ \frac{\GV_{d\beta}^{(3)}}{r^2d^2} }{1-\zeta q} \Bigg) + Q^{d\beta}\Bigg(  \frac{\GV_{d\beta}^{(2)}}{2d} - \frac{\GV_{d\beta}^{(3)}}{2d^2}  + \frac{ \frac{e^{\beta_t}-1}{\beta_t} \frac{\GV_{d\beta}^{(3)}}{d^2} }{1-q} \Bigg)\Bigg].
\end{split}
\]
\end{lemma}
\begin{lemma} \label{lemma_d_expansion}
\[
\begin{split}
   &(1-q) \Phi^{01} \sum_{0\neq \beta \in H_2(X, \mathbb{Z})} \sum_{r=1}^{\infty} \GV_{\beta} Q^{r\beta} d(r,q,\beta_t)
   \\
   & = \Phi^{01} \sum_{{\rm ind}(\beta)=1} \sum_{d=1}^{\infty} \Bigg[ \sum_{r\geq 2}\sum_{\zeta^r=1} Q^{rd\beta} \beta_t\Bigg( \frac{ \frac{\GV^{(1)}_{d\beta}}{rd} + \frac{\GV_{d\beta}^{(3)}}{r^3d^3} }{1-\zeta q} + \frac{ \frac{-\GV_{d\beta}^{(3)}}{r^3d^3} }{(1-\zeta q)^2}  \Bigg) 
    \\
    &\hspace{40mm}  + Q^{d\beta} \Bigg(  f_0(d,\beta_t)  + \frac{f_1(d,\beta_t)}{1-q} +  \frac{f_2(d,\beta_t) }{(1-q)^2} \Bigg) \Bigg].
\end{split}   
\] 
where
\[
\begin{split}
    f_0(d,\beta_t) &= \beta_t \Big(\frac{\GV_{d\beta}^{(0)}}{2} - \frac{5\GV_{d\beta}^{(1)}}{12d} - \frac{ \GV_{d\beta}^{(3)} }{12d^3} \Big);
    \\
    f_1(d,\beta_t) &= \frac{\beta_t \GV_{d\beta}^{(1)}}{d} + \frac{ \beta_t^2 \GV_{d\beta}^{(2)} }{2d^2} + \Big( (3-\beta_t)e^{\beta_t}-3-\beta_t - \frac{\beta_t^2}{2} \Big)\frac{ \GV_{d\beta}^{(3)}}{d^3};
    \\
    f_2(d,\beta_t) &= \Big(e^{\beta_t}(\beta_t-2) + 2 \Big)\frac{ \GV_{d\beta}^{(3)} }{d^3}.
\end{split}
\]
\end{lemma}

\subsection{arm, leg, and tail} \label{s:3.4}
\begin{lemma} \label{lemma_arm_expansion}
\[
\begin{split}
    \widetilde{\arm} (q) = \sum_{{\rm ind}(\beta)=1} \sum_{d=1}^{\infty} Q^{d\beta} &\Bigg[ \sum_{j=1}^{n_1} \Phi^{1j} \beta_j  \Bigg( \GV_{d\beta}^{(1)} - \frac{\GV_{d\beta}^{(3)}}{d^2}  + \frac{\beta_t\GV_{d\beta}^{(2)}}{2d} - \frac{\beta_t\GV_{d\beta}^{(3)}} {2d^2} \Bigg)
    \\
    + &\Phi^{01} \Bigg( d\GV_{d\beta}^{(-1)} - \frac{\GV_{d\beta}^{(3)}}{d^3} + f_0(d,\beta_t) \Bigg) \Bigg] + O(1-q) .
\end{split}
\] 
In particular, $\ch(\widetilde{\arm} (q)) \in H^{\geq 4}(X)[\![1-q,Q]\!]$.
\end{lemma}
\begin{proof}
Note that 
\[
\begin{split}
    & \pi_+^{\fake} \Bigg( \zeta^{-1}\alpha + \frac{ (1-\zeta^{-1}\alpha + \zeta^{-1}\beta ) }{1-\zeta q} + \frac{ (1-\zeta^{-1})\beta + \zeta^{-1} \gamma }{(1-\zeta q)^2} + \frac{ (1-\zeta^{-1}) \gamma }{(1-\zeta q)^3} \Bigg)
    \\
    & = \pi_+^{\fake} \Bigg( (1-q) \Big(  \frac{\alpha}{1-\zeta q} + \frac{\beta}{(1-\zeta q)^2} + \frac{\gamma}{(1-\zeta q)^3}   \Big) \Bigg)
    \\
    & = \alpha \delta_{1,\zeta} + O(1-q).
\end{split}
\]
The lemma follows from Lemma~\ref{lemma_a_expansion}, Lemma~\ref{lemma_b_expansion}, Lemma~\ref{lemma_c_expansion}, and Lemma~\ref{lemma_d_expansion} by taking only $\zeta=1$ terms.
\end{proof}
\begin{lemma} \label{lemma_leg_expansion}
\[
\begin{split}
    \widetilde{\leg}_r (q) &:= \Psi^r \Big( \widetilde{\arm} (q)|_{t=0} \Big)
    \\
    & = \sum_{{\rm ind}(\beta)=1} \sum_{d=1}^{\infty} Q^{rd\beta} \Bigg[ \sum_{j=1}^{n_1}  r^2\Phi^{1j} \beta_j \Bigg( \GV_{d\beta}^{(1)} - \frac{\GV_{d\beta}^{(3)}}{d^2} \Bigg)
    \\
    &\hspace{35mm} + r^3 \Phi^{01} \Bigg( d\GV_{d\beta}^{(-1)} - \frac{\GV_{d\beta}^{(3)}}{d^3}  \Bigg) \Bigg] + O(1-q).
\end{split}
\]
In particular, $\ch(\widetilde{\leg}_r (q)) \in H^{\geq 4}(X)[\![1-q,Q]\!]$.
\end{lemma}
\begin{proof}
    This follows directly from 
    \[
    \Psi^r(\Phi^{ij}) := \ch^{-1} \ \Psi^r \ \ch (\Phi^{ij}) = r^{3-i} \Phi^{ij}
    \]
    for $i\leq 2$, since $\ch(\Phi^{ij}) \in H^{2(3-i)}(X)$ for $i\leq 2$ by Lemma~\ref{lemma_degPhi}.
\end{proof}
We won't need any explicit expression for $\widetilde{\tail}$ except for the fact that it has deg$_{\mathbb{C}} \geq 2$.
\begin{lemma} \label{lemma_tail_expansion}
    $\ch( \widetilde{\tail_{\zeta}}(q) ) \in H^{\geq 4}(X)[\![1-\zeta q,Q]\!]$.
\end{lemma}
\begin{proof}
    It follows from the definition of $\widetilde{\tail}$ and the observation that
    \[
    \ch \Big( \pi_+^{1-\zeta q}( \widetilde{J} - (1-q) - t )  \Big)\in H^{\geq 4}(X)[\![ 1-\zeta q, Q]\!].
    \]
\end{proof}

\subsection{Verification of the adelic characterization} \label{s:3.5}
We show that $\tilde{J}$ satisfy Theorem~\ref{thm_adelic_charac} and hence lies on the $K$-theoretic Lagrangian cone.

\begin{proposition} \label{prop_J_q=1}
\[
\tilde{J}|_{q=1} = J^{\fake}(t + \widetilde{\arm}).
\]
\end{proposition}
\begin{proof}
Note that
\[
\begin{split}
    &\tilde{J}|_{q=1} - \Bigg( (1-q) + t +   \frac{ \sum_{j=1}^{n_1}\Phi^{1j} \partial_{t_j}F(t)  }{1-q} + \frac{q\ \Phi^{01} F(t)}{(1-q)^2} \Bigg) 
    \\
    &=\widetilde{\arm} (q) + \sum_{{\rm ind}(\beta)=1} \sum_{d=1}^{\infty} Q^{d\beta} \cdot \Bigg[ \sum_{j=1}^{n_1} \Phi^{1j} \beta_j \Bigg( \frac{ e^{\beta_t} \frac{\GV_{d\beta}^{(3)}}{d^2} }{1-q} \Bigg) 
    \\
    &\hspace{40mm} + \Phi^{01} \Bigg( \frac{ f_1(d,\beta_t) + \frac{3\GV_{d\beta}^{(3)}}{d^3} }{1-q} + \frac{ f_2(d,\beta_t) -\frac{2\GV_{d\beta}^{(3)}}{d^3} }{(1-q)^2} \Bigg) \Bigg],
\end{split}
\]
and that 
\[
\begin{split}
    & J^{\fake}(t) - \Bigg( (1-q) + t +   \frac{ \sum_{j=1}^{n_1}\Phi^{1j} \partial_{t_j}F(t)  }{1-q} + \frac{q\ \Phi^{01} F(t)}{(1-q)^2} \Bigg)
    \\
    &= \sum_{\beta \in H_2(X,\mathbb{Z})}  \GW_{\beta} Q^{\beta} e^{\beta_t} \Bigg[ \sum_{j=1}^{n_1} \beta_j\Phi^{1j} \Bigg( \frac{1}{1-q}\Bigg) + \Phi^{01} \Bigg( \frac{3-\beta_t}{1-q} + \frac{\beta_t -2 }{(1-q)^2} \Bigg) \Bigg]
    \\
    &= \sum_{{\rm ind}(\beta)=1} \sum_{d\geq 1} \frac{\GV_{d\beta}^{(3)}}{d^3} Q^{d\beta} e^{\beta_t} \Bigg[ \sum_{j=1}^{n_1} \beta_j \Phi^{1j} \Bigg( \frac{d}{1-q}  \Bigg) + \Phi^{01} \Bigg( \frac{3-\beta_t}{1-q} + \frac{\beta_t -2 }{(1-q)^2} \Bigg)  \Bigg].
\end{split}
\]
Their difference comes from the $\widetilde{\arm}$ contribution:
\[
\begin{split}
    \tilde{J}|_{q=1} - J^{\fake}(t) &= \widetilde{\arm} (q) + \sum_{{\rm ind}(\beta)=1} \sum_{d=1}^{\infty} \Phi^{01} \Bigg( \frac{f_1(d,\beta_t) + \frac{3\GV_{d\beta}^{(3)}}{d^3} }{1-q} \Bigg) 
    \\
    &= J^{\fake}(t+ \widetilde{\arm}) - J^{\fake}(t).
\end{split}
\]
The last equality follows from the Poincar\'e pairing:
\[
\langle \frac{1}{1-qL}, \widetilde{\arm} (L), t \rangle_{0,3,\beta'=0} = \sum_{{\rm ind}(\beta)=1} \sum_{d=1}^{\infty} \Bigg( \frac{f_1(d,\beta_t) + \frac{3\GV_{d\beta}^{(3)}}{d^3} }{1-q} \Bigg),
\]
since $\ch(\widetilde{\arm} (q)) \in H^{\geq 4}(X)[\![1- q,Q]\!]$ by Lemma~\ref{lemma_arm_expansion} and hence all terms with $\beta'\neq 0$ vanish by Lemma~\ref{lemma_vanishingdeg2}. This completes the proof.
\end{proof}
We divide the computation of $\tilde{J}|_{q=\zeta^{-1}}$ into two parts, $t=0$ case and $t\neq 0$ case.
\begin{proposition}[{\cite[\S~4]{Chou_Lee_2022_QKGVI}}]  \label{prop:Jzeta_t=0}  Let $t=0$ and $\zeta$ be primitive $r$-th roots of unity. We have
\[
\begin{split}
    \tilde{J}|_{q=\zeta^{-1}} &= \tilde{\delta}_{\zeta}(\zeta q)
    \\
    &+ \sum_{n,\beta,\alpha} \frac{Q^{r\beta} \Phi_{\alpha}}{n!} \Big[ \frac{\Phi^{\alpha}}{1-\zeta q L^{1/r}}, \leg_{r}(L) ,\dots, \leg_{r}(L), \tilde{\delta}_{\zeta}(L^{1/r}) \Big]^{X_{\zeta}}_{0,n+2,\beta}.
\end{split}
\]
\end{proposition}
\begin{proof}
The computation is similar to the section 4 in \cite{Chou_Lee_2022_QKGVI} for $X$ the quintic threefold. We sketch the proof here.

From the definition of $\widetilde{\tail}_{\zeta}$, the polynomial term of $\tilde{J}|_{q=\zeta^{-1}}$ equals 
\[
1-q + \widetilde{\tail}_{\zeta}(q) = (1-\zeta^{-1}) + \zeta^{-1}(1-\zeta q) +\widetilde{\tail}_{\zeta}(q)  = \tilde{\delta}_{\zeta}(\zeta q) \in \mathbf{K}[\![1-\zeta q]\!].
\] 

We use Lemma~\ref{lemma_a_expansion}, \ref{lemma_b_expansion}, \ref{lemma_c_expansion}, and \ref{lemma_d_expansion} for the principal part of LHS. The principal part on the RHS equals
\[
\begin{split}
    &\sum_{\alpha} \Phi_{\alpha} \Bigg( \Big[ \frac{\Phi^{\alpha}}{1-\zeta qL^{1/r} }, \leg_r(L), 1-\zeta^{-1}L^{1/r} + \widetilde{\tail}_{\zeta}(\zeta^{-1}L^{1/r})\Big]^{X_{\zeta}}_{0,3,\beta=0}
    \\
    & \hspace{20mm} + \sum_{\beta} Q^{r\beta}\Big[ \frac{\Phi^{\alpha}}{1-\zeta q L^{1/r}}, 1-\zeta^{-1} L^{1/r} + \widetilde{\tail}_{\zeta}(\zeta^{-1}L^{1/r})\Big]^{X_{\zeta}}_{0,2,\beta} \Bigg). 
\end{split}
\]
Since $\ch(\leg_r(L)) \in H^{\geq 4}(X)[\![1-q,Q]\!])$ by Lemma~\ref{lemma_leg_expansion}, only $\beta=0$ and $n=3$ term for the first expression has nonzero contribution. It can be computed directly
\begin{equation} \label{eqn_t=0_leg}
\begin{split}
    &\sum_{\alpha} \Phi_{\alpha} \Big[ \frac{\Phi^{\alpha}}{1-\zeta qL^{1/r}}, \leg_r(L), 1-\zeta^{-1}L^{1/r} + \widetilde{\tail}_{\zeta}(\zeta^{-1}L^{1/r})\Big]^{X_{\zeta}}_{0,3,0}
    \\
    &= \frac{1-\zeta^{-1}}{r^2(1 - \zeta q)} \sum_{\alpha} \Phi_{\alpha}\langle \Phi^{\alpha}, \leg_r(L), 1 \rangle^X_{0,3,0}
    \\
    & = \frac{1-\zeta^{-1}}{1-\zeta q}\sum_{\ind(\beta)=1}\sum_{d=1}^{\infty} Q^{rd\beta} 
    \\
    & \hspace{20mm} \cdot\Bigg[ \sum_{j=1}^{n_1}  \Phi^{1j} \beta_j \Bigg( \GV_{d\beta}^{(1)} - \frac{\GV_{d\beta}^{(3)}}{d^3} \Bigg) + r\Phi^{01} \Bigg( d\GV_{d\beta}^{(-1)} - \frac{\GV_{d\beta}^{(3)}}{d^3}  \Bigg) \Bigg].
\end{split}
\end{equation}
In the above computation, all terms involving $\widetilde{\tail}$ vanish by Lemma~\ref{lemma_vanishingdeg2}.
The $r^{-2}$ after the first equality follows from the constant term of the twisting class by Lemma~\ref{lemma_rpower} and the Poincar\'e pairing on $X/ \mathbb{Z}_r$. 

We compute the second expression as follows:
\[ 
\begin{split}
    &\sum_{\beta} Q^{r\beta}\Big[ \frac{\Phi^{\alpha}}{1-\zeta q L^{1/r}}, 1-\zeta^{-1} L^{1/r} +  \widetilde{\tail}_{\zeta}(\zeta^{-1}L^{1/r})\Big]^{X_{\zeta}}_{0,2,\beta} 
    \\
    &=  \frac{1}{r}\sum_{\alpha}\Phi^{\alpha} \langle \frac{\Phi_{\alpha}}{1-\zeta q L^{1/r}}, \Psi^r ( \sum_{\alpha}\Phi^{\alpha}\langle \Phi_{\alpha}\rangle^{\fake} ) , 1-\zeta^{-1} L^{1/r} \rangle^{X/\mathbb{Z}_r, K}_{0,3,0}
    \\
    &\hspace{25mm} +\sum_{\alpha} \sum_{\beta} Q^{r\beta} \Phi_{\alpha} \langle \frac{\Phi^{\alpha}}{1-\zeta q L^{1/r}}, 1-\zeta^{-1} L^{1/r}\rangle^{ X/ \mathbb{Z}_r, K }_{0,2,r\beta}.
\end{split}
\]
All terms involving $\widetilde{\tail}$ vanish by Lemma~\ref{lemma_vanishingdeg2}.
The term involving $\langle \rangle^{\fake}$ comes from twisting of type $C$. Two expressions can be computed directly: 
\begin{equation} \label{eqn_t=0_fake}
\begin{split}
    & \frac{1}{r} \sum_{\alpha} \Phi^{\alpha} \langle \frac{\Phi_{\alpha}}{1-\zeta q}, \Psi^r ( \sum_{\alpha'} \Phi^{\alpha'} \langle \Phi_{\alpha'} \rangle^{\fake} ), 1 \rangle_{0,3,0}^{X/ \mathbb{Z}_r,K}
    \\
    & = \frac{1-\zeta^{-1}}{r^2(1-\zeta q)}\sum_{\alpha} \Phi^{\alpha}\langle \Phi_{\alpha}, \sum_{\beta>0} Q^{r\beta}\Big( \sum_{j=1}^{n_1}r^2\Phi^{1j}\beta_j \GW_{\beta} + r^3 \Phi^{01} \GW_{\beta} \Big), 1 \rangle_{0,3,0}^{X,K}
    \\
    &= \frac{1-\zeta^{-1}}{1-\zeta q}\sum_{\ind(\beta)=1}\sum_{d=1}^{\infty}\Bigg[\sum_{j=1}^{n_1} \Phi^{1j}\beta_j \Bigg(\frac{\GV_{d\beta}^{(3)}}{d^3}\Big) + r \Phi^{01} \Bigg(\frac{\GV_{d\beta}^{(3)}}{d^3} \Bigg) \Bigg].
\end{split}
\end{equation}
\begin{equation} \label{eqn_t=0_main}
\begin{split}
    & \sum_{\alpha} \sum_{\beta} Q^{r\beta} \Phi^{\alpha} \langle \frac{\Phi_{\alpha}}{1-\zeta q L^{1/r}}, 1-\zeta^{-1} L^{1/r}\rangle^{ X/ \mathbb{Z}_r,K }_{0,2,r\beta}
    \\
    & =\sum_{\ind(\beta)=1} \sum_{d=1}^{\infty} Q^{rd\beta}
    \\
    &\cdot\Bigg[ \sum_{j=1}^{n_1} \Phi^{1j} \beta_j \Bigg( \frac{(1-\zeta^{-1})(\frac{-\GV_{d\beta}^{(3)}}{r^2d^2}) + \zeta^{-1} (\frac{\GV^{(3)}_{d\beta}}{r^2d^2} ) }{1-\zeta q} + \frac{(1-\zeta^{-1}) (\frac{\GV_{d\beta}^{(3)}}{r^2d^2}) }{(1-\zeta q)^2} \Bigg)
    \\
    &  + \Phi^{01} \Bigg(  \frac{ (1-\zeta^{-1})(\frac{-\GV^{(3)}_{d\beta} }{r^3d^3}) +\zeta^{-1} (\frac{3\GV_{d\beta}^{(3)}}{r^3d^3} ) }{1-\zeta q} + \frac{ (1-\zeta^{-1})(\frac{\GV^{(3)}_{d\beta}}{r^3d^3}) + \zeta^{-1} (\frac{-2\GV_{d\beta}^{(3)}}{r^3d^3} ) }{(1-\zeta q)^2} 
    \\
    & \hspace{80mm} + \frac{(1-\zeta^{-1})(\frac{-2\GV_{d\beta}^{(3)}}{r^3d^3}) }{(1-\zeta q)^3}\Bigg)  \Bigg].
\end{split}
\end{equation}

Combine equation~(\ref{eqn_t=0_fake}), (\ref{eqn_t=0_leg}), and (\ref{eqn_t=0_main}) gives the principal part of RHS and it coincides with the principal part of LHS. This completes the proof.
\end{proof}

\begin{proposition} \label{prop_Jzeta_tnot0} Let $t= \sum_{j=1}^{n_1} t_j \Phi_{1j}$ and $\zeta$ be primitive $r$-th roots of unity. We have
\[
\begin{split}
    \tilde{J}_{q=\zeta^{-1}} &= \tilde{\delta}_{\zeta}(\zeta q)
    \\
    &+ \sum_{n,\beta,\alpha} \frac{Q^{r\beta} \Phi_{\alpha}}{n!} \Big[ \frac{\Phi^{\alpha}}{1-\zeta q L^{1/r}}, \leg_{r}(L) ,\dots, \leg_{r}(L), \tilde{\delta}_{\zeta}(L^{1/r}) \Big]^{X_{\zeta}}_{0,n+2,\beta},
\end{split}
\]
\end{proposition}
\begin{proof}
The computation for $\beta_t=0$ part has been covered by previous lemma. We compare the $\beta_t\neq 0$ part. On the LHS, we consider only $c(r,q,\beta_t)$ and $d(r,q,\beta_t)$ terms. On the RHS, only need to consider 
\[
\begin{split}
    \sum_{n,\beta,\alpha} \frac{Q^{r\beta} \Phi_{\alpha}}{n!} \Big[ \frac{\Phi^{\alpha}}{1-\zeta q L^{1/r}}, \leg_{r}(L) ,\dots, \leg_{r}(L), t \Big]^{X_{\zeta}}_{0,n+2,\beta}.
\end{split}
\]
since $\leg_{\zeta}(L)$ is independent of $t$.
With the same argument as in Proposition~\ref{prop:Jzeta_t=0}, we compute the following three expressions:
\begin{equation} \label{eqn:tnot0_leg}
\begin{split}
     \sum_{\alpha} \Phi_{\alpha} \Big[ \frac{\Phi^{\alpha}}{1-\zeta qL^{1/r}}, \leg_r(L), t \Big]^{X_{\zeta}}_{0,3,0}
    = \Phi^{01} \beta_t\Bigg( \frac{ \GV_{d\beta}^{(1)} - \frac{\GV_{d\beta}^{(3)}}{r^3d^3} }{1-\zeta q} \Bigg)
\end{split}
\end{equation}
\begin{equation} \label{eqn_tnot0_fake}
\begin{split}
    & \frac{1}{r} \sum_{\alpha} \Phi^{\alpha} \langle \frac{\Phi_{\alpha}}{1-\zeta q}, \Phi^r ( \sum_{\alpha} \Phi^{\alpha} \langle \Phi_{\alpha} \rangle^{\fake} ), t \rangle_{0,3,0}^{X/ \mathbb{Z}_r, K} = \Phi^{01} \beta_t \Bigg( \frac{ \frac{\GV_{d\beta}^{(3)}}{r^3d^3} }{1-\zeta q} \Bigg)
\end{split}
\end{equation}
\begin{equation} \label{eqn_tnot0_main}
\begin{split}
    &\sum_{\alpha} \sum_{\beta} Q^{r\beta} \Phi_{\alpha} \langle \frac{\Phi^{\alpha}}{1-\zeta q L^{1/r}}, t \rangle^{ X/ \mathbb{Z}_r, K }_{0,2,r\beta}
    \\
    & = \sum_{j=1}^{n_1} \beta_j\Phi^{1j} \beta_t \Bigg( \frac{ \frac{\GV_{d\beta}^{(3)}}{r^2d^2} }{1-\zeta q} \Bigg) + \Phi^{01} \beta_t\Bigg( 
    \frac{ \frac{\GV_{d\beta}^{(3)}}{r^3d^3} }{1-\zeta q} - \frac{ \frac{\GV_{d\beta}^{(3)}}{r^3d^3} }{(1-\zeta q)^2} \Bigg)
\end{split}
\end{equation}
Combine equation~(\ref{eqn:tnot0_leg}), (\ref{eqn_tnot0_fake}), and (\ref{eqn_tnot0_main}) gives the RHS and it coincides with LHS.

This completes the proof.
\end{proof}

\begin{theorem}[ = Theorem~\ref{conjecture_npt}] \label{thm_conj}
    $\tilde{J} \in \mathcal{L}^{K}$. Moreover, $\tilde{J} = J^K(t)$. 
\end{theorem}
\begin{proof}
For the first statement $\tilde{J} \in \mathcal{L}^{K}$, we use adelic characterizaion theorem, Theorem~\ref{thm_adelic_charac}. It suffices to check the following three criteria:
\begin{itemize}
    \item[(i)] $\tilde{J}$ only has poles at roots of unity. This criterion follows immediately from the expression of $\tilde{J}$
    \item[(ii)] The fake theory computation follows from Lemma~\ref{prop_J_q=1}.
    \item[(iii)] The stem theory computation follows from Proposition~\ref{prop_Jzeta_tnot0}.
\end{itemize}
This proves the first statement. The second statement follows from Givental's uniqueness theorem. Since every function $f$ whose range lies on the Lagrangian cone is completely determined by $\pi_+(f)$ and
\[
\pi_+(\tilde{J}) = 1- q + t = \pi_+ ( J^K(t) ) ,
\]
we conclude that $\tilde{J} = J^K(t)$. 
\end{proof}

\begin{remark} \label{r:3.14}

The dependence of $J^K$ on $\beta_t$ is of the form $e^{\beta_t}$, except finitely many monomials in $\beta_t^{\leq 2}$ and an overall linear factor $\beta_t$ to some terms. This can be explained in geometric terms.

Recall that for GW on CY3, invariants with (even) classes of $\deg_{\mathbb{C}} \geq 2$ vanish. The divisor and string equations ``remove'' $\deg_{\mathbb{C}} =1$ and $0$ classes in the insertions, and reduce all invariants to $0$-pointed  $\langle \cdot \rangle_{g, n=0,\beta}$. Furthermore, the dependence on the $\deg_{\mathbb{C}}=1$ classes $D$ is of the form $e^{\int_{\beta} D}$. A closed formula of this reduction can be found in \cite{Fan_Lee_2019}. 

If the moduli of stable maps at genus zero were virtually smooth \emph{schemes}, the virtual HRR would imply that quantum $K$-invariants behave almost like GW invariants. As we are restricting the insertions in $J^K$ to $t \in \op{ch}^{-1} H^2(X)$, and the fact that the $n$-pointed moduli has virtual dimension $n$, reduction from $n$-point to $1$-point would be simply applying the divisor axiom in GW and the result should depend on $e^{\beta_t}$. The extra linear overall factor to some terms also comes from the divisor axiom.

This is however not true, but the stacky contributions \cite[Theorem~5.1]{Ediddin_RR} in this case is limited to finitely many graphs. This implies the simple dependence of $J^K (t)$ on $\beta_t$. For completeness, we list the possible ''Kawasaki strata'' below.

The constant term ($\beta_t=0$) comes from the graph sum of the Kawasaki strata in Figure~\ref{Graph_Kawasaki_Strata_beta0}. The linear terms  of $\beta_t$ and higher order terms are given by Figure~\ref{Graph_Kawasaki_Strata_beta1} and Figure~\ref{Graph_Kawasaki_Strata_beta2} respectively.
    \begin{figure}[ht]
    \centering
    \begin{tikzpicture}[scale=1.5]
    \filldraw[black] (-3,0.75) node[anchor=west]{$\frac{\Phi_{\alpha}}{1-\zeta_r q L^{1/r}}$};
    \filldraw[black] (-4,0.2) node[anchor=west]{$\displaystyle\sum_{r\geq 2} \sum_{\zeta_r}$};
    \filldraw[-to] (-3,0)--(-3, 0.8);
    \filldraw[ultra thick] (-3,0)--(-2,0);
    \filldraw[black] (-2,0) circle(1.5pt);
    \filldraw[black] (-2.1,0.2) node[anchor=west]{$\delta_{\zeta_r}$};
    \draw (-3.1, 1) .. controls (-3.2, 0.4) .. (-3.1,-0.2);
    \filldraw[black] (-1.5,0.4) node[anchor=west]{$+$};
    \filldraw[ultra thick] (0, 0.9)--(0, 0);
    \filldraw[-to] (0,0.9)--(-0.8,0.9);
    \filldraw[black] (-1.1, 0.65) node[anchor=west]{$\frac{\Phi_{\alpha}}{1-\zeta_r q L^{1/r}} $ };
    \draw (0,0.4).. controls (0.25, 0.5) and (0.5, 0.3) ..(0.7,0.4);
    \filldraw[black] (0,0) circle(1.5pt);
    \filldraw[black] (0,0.6) node[anchor=west]{$\leg_r(L)$};
    \filldraw[black] (0,0) node[anchor=west]{$\delta_{\zeta_r}$};
    \draw (1, 1) .. controls (1.1, 0.4) .. (1, -0.2);
    \filldraw[black] (1.1,0.4) node[anchor=west]{$+$};
    \draw (2,0) .. controls (2.6,0.5) and (3.3,-0.5) .. (4,0);
    \filldraw[black] (2.5,0.3) node[anchor=west]{$\arm(L)|_{t=0}$};
    \filldraw[thick,-to](2,0)--(2,0.8);
    \filldraw[black] (1.3,0.8) node[anchor=west]{$\frac{\Phi_{\alpha}}{1-qL}$};
    \end{tikzpicture}
    \caption{Kawasaki strata, $\beta_t^0$}
    \label{Graph_Kawasaki_Strata_beta0}
    \end{figure}
    
    \begin{figure}[ht]
    \centering
    \begin{tikzpicture}[scale=1.5]
    \filldraw[black] (-3,0.75) node[anchor=west]{$\frac{\Phi_{\alpha}}{1-\zeta_r q L^{1/r}}$};
    \filldraw[black] (-4,0.2) node[anchor=west]{$\displaystyle\sum_{r\geq 2} \sum_{\zeta_r}$};
    \filldraw[-to] (-3,0)--(-3, 0.8);
    \filldraw[ultra thick] (-3,0)--(-2,0);
    \filldraw[black] (-2,0) circle(1.5pt);
    \filldraw[black] (-2.1,0.2) node[anchor=west]{$t$};
    \draw (-3.1, 1) .. controls (-3.2, 0.4) .. (-3.1,-0.2);
    \filldraw[black] (-1.5,0.4) node[anchor=west]{$+$};
    \filldraw[ultra thick] (0, 0.9)--(0, 0);
    \filldraw[-to] (0,0.9)--(-0.8,0.9);
    \filldraw[black] (-1.1, 0.65) node[anchor=west]{$\frac{\Phi_{\alpha}}{1-\zeta_r q L^{1/r}} $ };
    \draw (0,0.4).. controls (0.25, 0.5) and (0.5, 0.3) ..(0.7,0.4);
    \filldraw[black] (0,0) circle(1.5pt);
    \filldraw[black] (0,0.6) node[anchor=west]{$\leg_r(L)$};
    \filldraw[black] (0,0) node[anchor=west]{$t$};
    \draw (1, 1) .. controls (1.1, 0.4) .. (1, -0.2);
    \filldraw[black] (-3, -1.1) node[anchor=west]{ $+$ };
    \filldraw[thick] (2.3, -0.6)--(2.3, -1.5);
    \filldraw[-to] (2.3,-0.6)--(1.5,-0.6);
    \filldraw[black] (1.2, -0.85) node[anchor=west]{$\frac{\Phi_{\alpha}}{1- q L} $ };
    \draw (2.3,-1.05) .. controls (2.7, -0.75) and (3.1, -1.35) ..  (3.5, -1.05);
    \filldraw[black] (2.5,-0.8) node[anchor=west]{$\arm(L)|_{t=0}$};
    \filldraw[black] (2.3,-1.5) circle(1.5pt);
    \filldraw[black] (2.3,-1.5) node[anchor=west]{$t$};
    \filldraw[black] (0.4,-1.1) node[anchor=west]{$+$};
    \draw (-1.8,-1.5) .. controls (-1.2,-1) and (-0.5,-2) .. (0.2,-1.5);
    \filldraw[thick,-to](-1.8,-1.5)--(-1.8,-0.7);
    \filldraw[black] (-2.5,-0.7) node[anchor=west]{$\frac{\Phi_{\alpha}}{1-qL}$};
    \filldraw[black] (-1.3,-1.37) circle(1.5pt);
    \filldraw[black] (-1.4,-1.17) node[anchor=west]{$t$};
    \end{tikzpicture}
    \caption{Kawasaki strata, $\beta_t^1$}
    \label{Graph_Kawasaki_Strata_beta1}
    \end{figure}

    \begin{figure}[ht]
    \centering
    \begin{tikzpicture}[scale=1.5]
    \filldraw[thick] (2.3, -0.6)--(2.3, -1.5);
    \filldraw[-to] (2.3,-0.6)--(1.5,-0.6);
    \filldraw[black] (1.2, -0.85) node[anchor=west]{$\frac{\Phi_{\alpha}}{1- q L} $ };
    \draw (2.3,-1.05) .. controls (2.7, -0.75) and (3.1, -1.35) ..  (3.5, -1.05);
    \filldraw[black] (2.4,-0.8) node[anchor=west]{$\arm(L)$};
    \filldraw[black] (2.3,-1.5) circle(1.5pt);
    \filldraw[black] (2.3,-1.5) node[anchor=west]{$t$};
    \filldraw[black] (0.4,-1.1) node[anchor=west]{$+$};
    \draw (-1.8,-1.5) .. controls (-1.2,-1) and (-0.5,-2) .. (0.2,-1.5);
    \filldraw[thick,-to](-1.8,-1.5)--(-1.8,-0.7);
    \filldraw[black] (-2.5,-0.7) node[anchor=west]{$\frac{\Phi_{\alpha}}{1-qL}$};
    \filldraw[black] (-1.3,-1.37) circle(1.5pt);
    \filldraw[black] (-1.4,-1.17) node[anchor=west]{$t$};
    \filldraw[black] (-1,-1.45) circle(1.5pt);
    \filldraw[black] (-1.1,-1.25) node[anchor=west]{$t$};
    \filldraw[black] (0,-1.6) circle(1.5pt);
    \filldraw[black] (-0.1,-1.4) node[anchor=west]{$t$};
    \filldraw[black] (-0.7,-1.3) node[anchor=west]{$\dots$};
    \end{tikzpicture}
    \caption{Kawasaki strata, $\beta_t^{\geq 2}$}
    \label{Graph_Kawasaki_Strata_beta2}
    \end{figure}

Notations are explained below. See \cite{Chou_Lee_2022_QKGVI} or \cite{Givental_Tonita_2011} for more detail. 
\begin{itemize}
    \item The black points denote the marked points.
    \item $\delta_{\zeta}:= \delta_{\zeta}(J^K(0);q)$, and $\arm(q):= \arm(J^K(t);q)$.
    \item Given a stable map with symmetry $g$, the action of $g$ fixes the marked points and acts on $L_1$ with an eigenvalue, which is denote by $\zeta_r$. $\displaystyle \sum_{\zeta_r}$ denotes the sum over all primitive $r$-th roots of unity.
    \item The first marked point is denoted as an arrow with input $ \frac{\Phi_{\alpha}}{1-\zeta_r q L^{1/r}}$ for the stem theory and $\frac{\Phi_{\alpha}}{1-qL}$ for the fake theory.
    \item The vertical line denotes the curve that contracts to a point under the stable map. 
    \item The stem curve is emphasize as a thick line.
\end{itemize}
Note that $\arm(L)$ only has monomials in $\beta_t^{\leq 1}$ since stem theory only have that.  Hence the right graph in Figure~\ref{Graph_Kawasaki_Strata_beta2} contributes terms with $\beta_t^{\leq 2}$. All $\beta_t^{\geq 3}$ terms comes only from the left graph in Figure~\ref{Graph_Kawasaki_Strata_beta2} and the virtual Hirzebruch--Riemann--Roch for \emph{schemes} applies. For those, the usual divisor axiom in Gromov--Witten theory gives $e^{\beta_t}$ as explained above.
\end{remark}

\section{Multiple cover formula and virtual Clemens conjecture}
In this section, we study the $K$-theoretic multiple cover formula at genus zero and give a heuristic derivation of 
Theorem~\ref{conjecture_npt} based on the ``Virtual Clemens Conjecture'', a folklore ansatz for enumerative geometry on Calabi--Yau threefolds. 

We note that Jockers and Mayr have already had a similar computation for the ``resolved conifold'' in \cite[\S~2.3]{Jockers_Mayr_2019} where the equivariant quantum $K$-theory was employed.

\subsection{Multiple cover formula}
\begin{lemma} \label{l:4.4}
For the total space $X_{-1,-1}$ of $\O(-1) \oplus \O (-1) \to P^1$, the Gopakumar--Vafa invariants $\op{GV}_{0,1} =1$ (genus zero and degree $1$) and $\op{GV}_{g,d} =0$ otherwise
\end{lemma}

\begin{proof}
This follows from the Gopakumar--Vafa equation \eqref{e:1.2}, and the results in Gromov--Witten theory: for $g=0$ the Voisin--Aspinwall--Morrison formula \cite{Voisin_1996}\cite{Aspinwall_Morrison_1993}
\[
 \op{GW}_{0,0,d} = \frac{1}{d^3},
\]
for $g=1$, the BCOV and Graber--Pandharipande formula \cite{Graber_Pandharipande_1997}
\[
 \op{GW}_{1,0,d} = \frac{1}{12 d},
\]
and for $g \geq 2$ the Faber--Pandharipande formula \cite{Faber_Pandharipande_1998}
\[
 \op{GW}_{g,0,d} = \frac{|B_{2g}| d^{2g-3}}{2g \cdot (2g-2)!}.
\]
\end{proof}

In order to consider the small $J$-function of $X_{-1,-1}$ in quantum $K$-theory, we consider its compatification by the ``infinity divisor''. Let 
\[
 Y_{-1,-1} := P_{P^1}(\mathcal{O}(-1)\oplus \mathcal{O}(-1)\oplus \mathcal{O} ).
 \]
Let $P = \pi^*\mathcal{O}(-1)$ with $\pi: Y_{-1,-1} \rightarrow P^1$, and $T = \mathcal{O}(D_{\infty})$ with $D_{\infty}\subset Y_{-1,-1}$, the infinity divisor. We denote $Q^r := Q^{r\ell}$, where 
\[
 \ell := [P^1] \xhookrightarrow{0} Y_{-1,-1}
\] 
is the line class in the ``zero section'' of the projective bundle. In the following, we consider the \emph{specialized} small $I$-function and $J$-function of $Y_{-1,-1}$ only for the curve classes in the zero section, i.e., multiples of $\ell$.

\begin{lemma} \label{l:4.5} 
The small $I$-function for 
$Y_{-1,-1}$, with curve classes in the zero section, is
\[
  I^{Y_{-1,-1}}(q, Q) = (1-q) \left[ 1+ \sum_{r=1}^{\infty} Q^r \frac{(1-PT)^2 \prod_{m=1}^{r-1}(1-PT q^m)^2}
  {(PT)^{2r}q^{r(r-1)} \prod_{m=1}^r (1-Pq^m)^2} \right].
\]
\end{lemma}

\begin{proof}
This lemma follows from the computations of the $I$-functions for toric manifolds via fixed point localization by A.~Givental and collaborators. See \cite{Givental_Tonita_2011, Givental_PEV_Toric_2015}. Their formula for small $I$-function gives
\[
  I^{Y_{-1,-1}}(q, Q) = (1-q) \left[ 1+ \sum_{r=1}^{\infty} Q^r \frac{\prod_{m={-r+1}}^{0}(1-P^{-1} T^{-1} q^m)^2}
  {\prod_{m=1}^r (1-Pq^m)^2} \right].
\]
A simple manipulation gives the above presentation.
\end{proof}

We note that in this case the small $I$-function includes a factor of $(1-PT)^2$, the $K$-theoretic normal bundle of $P^1$ embedded in $Y_{-1,-1}$. One may think of this as an $I$-function of a toric completion $P_{P^1} (\O(-1) \oplus \O (-1) \oplus \O)$, with curve class in the base $P^1$. 
This small $I$-function is different from the small $J$-function, as it has poles at $q=0$. 
(However, it satisfies $\lim_{q \to \infty} I^{X_{-1,-1}}(q) =0$.) Using this $I$-function one can obtain the small $J$-function via the ``\emph{generalized} mirror transform'' (also known as the explicit reconstruction, or Birkhoff factorization).

\begin{proposition} \label{prop_multicover_t=0}
The small $J$-function for 
$Y_{-1,-1}$, with curve classes in the zero section, is 
\begin{equation} \label{e:4.1}
\begin{split}
\frac{1}{1-q} & J^{Y_{-1,-1}}(q, Q) = 1 + (1-PT)^2\Big(1 +(1-P)\Big)\sum_{r \geq 1} \, Q^r a(r, q^r) 
\\
& + (1-PT)^2(1-P) \sum_{r\geq 1} Q^r \, b(r, q^r). 
\end{split} 
\end{equation}
\end{proposition}
\begin{proof} 
The $K$-ring of $Y_{-1,-1}$ has the following presentation
\[
 K(Y_{-1,-1}) = \frac{\mathbb{Z} [P, T]}{\left( (1-P)^2, (1-PT)^2(1-T) \right)}.
\]
We use the relations of $K(Y_{-1,-1})$ to rewrite $I^{Y_{-1,-1}}$ as follows. Since $(1-PT)^2(1-T) =0$, we have
\begin{equation} \label{e:4.2}
\begin{split}
   &(1-PT)^2 (1-Pq^m T) = (1-PT)^2 (1-Pq^m [1-(1-T)]) \\
   = &(1-PT)^2 (1-Pq^m).  
\end{split}
\end{equation}
Therefore,
\[
\begin{split}
    & I^{Y_{-1,-1}} - (1-q)\\
    = &(1-q) \sum_{r=1}^{\infty} Q^r \frac{(1-PT)^2} 
    {(PT)^{2r}q^{r(r-1)}(1-Pq^r)^2}
    \\
    = &(1-q)\sum_{r\geq 1}Q^r\left[ \Big( \sum_{i=1}^{r-1} \frac{i}{q^{r(r-i)}}\Big)(1-PT)^2  + \Big( \sum_{i=1}^{r-1} \frac{i(2r-i+1)}{q^{r(r-i)}} \Big) (1-PT)^3 \right]
    \\
    & + (1-q) \Big((1-PT)^2 +(1-PT)^3\Big)\sum_{r\geq 1} Q^r \Big( \frac{r-1}{1-q^r} + \frac{1}{(1-q^r)^2} \Big)
    \\
    & + (1-q)(1-PT)^3 \sum_{r\geq 1}Q^r\Big( \frac{r^2-1}{1-q^r} + \frac{3}{(1-q^r)^2} -\frac{2}{(1-q^r)^3} \Big).
\end{split}
\]
In the first equality, the factor $\prod_{m=1}^{r-1} (1-PTq^m)^2$ in the numerator and $\prod_{m=1}^{r-1}(1-Pq^m)^2$ in the denominator cancel each other due to the presence of $(1-PT)^2$ and \eqref{e:4.3}.
The second equality follows from an explicit computation.

The first line after the second equal sign lies in $\mathcal{K}_+$ and the rest lies in $\mathcal{K}_-$. Consider the reconstruction theorem \cite[Theorem~2]{Givental_PEVIII_ER}
\[
\begin{split}
    &J^{Y_{-1,-1}}(q, Q) =  \sum_{r\geq 0}  I_r^{Y_{-1,-1}} Q^r  
    \\ 
    & \cdot\exp \left( \frac{ \sum_{i=0}^1\delta_i(Q)(1-Pq^{r})(1-PT q^r)^i + \sum_{i=0}^3 \e_i(Q) (1-PT q^{r})^i}{(1-q)} \right) 
    \\
    &\qquad \cdot\left( \sum_{i=0}^1 s_i(q,Q)(1-Pq^r)(1-PT q^r)^i + \sum_{i=0}^3 u_i(q,Q) (1-PTq^r)^i \right),
\end{split}
\]
for some uniquely determined $\e_i(Q)$, $\delta_i(Q)$, $s_i(q,Q)$ and $u_i (q, Q)$, where
\[
\begin{split}
    \e_i(Q)& = \sum_{j \geq 1} \e_{ij}Q^j \in \mathbb{Q}[\![Q]\!], \\
    \delta_i(Q) &=\sum_{j \geq 1} \delta_{ij}Q^j \in \mathbb{Q}[\![Q]\!],
    \\
    u_i(q,Q) 
   & = \sum_{j\geq 0} u_{ij}(q) Q^j \in \mathbb{Q}[q,q^{-1}][\![Q]\!],
   \\
   s_i(q,Q) 
   & = \sum_{j\geq 0} s_{ij}(q) Q^j \in \mathbb{Q}[q,q^{-1}][\![Q]\!].
\end{split}
\]
A direct computation by induction on the degree of the Novikov variable shows that 
\[
u_0(q,Q) = 1, \quad u_1(q,Q)=\e_i(Q)=\delta_j(Q)=s_j(q,Q)=0,
\]
where $0\leq i \leq 3$ and $0\leq j \leq 1$. Note that $u_2(q,Q)$ and $u_3(q,Q)$ will not change the $\mathcal{K}_-$ part and that the $\mathcal{K}_-$ part coincides with the definition of $a(r,q^r)$ and $b(r,q^r)$.  This concludes the proof.
\end{proof}

We go through the same argument with input $t=t_1(1-P)$.
\begin{proposition} \label{prop_multicover_t}
The $J$-function for 
$Y_{-1,-1}$, with curve classes in the zero section, is 
\begin{equation} \label{e:4.3}
\begin{split}
\frac{1}{1-q} & J^{Y_{-1,-1}}(t, q, Q) = 1 + \frac{t_1(1-P)}{1-q}
\\
&+  (1-PT)^2\Big(1 +(1-P)\Big)\sum_{r \geq 1} \, Q^r\Big( a(r, q^r) + c(r,q, t_1) \Big) 
\\
& + (1-PT)^2(1-P) \sum_{r\geq 1} Q^r \, \Big( b(r, q^r) + d(r,q, t_1) \Big) . 
\end{split} 
\end{equation}
\end{proposition}
\begin{proof}
We compute $J^{Y_{-1,-1}}$ in two steps. Consider the flow:
\[
\begin{split}
    J^{Y_{-1,-1}}(t^*,q,Q) =\sum_{r\geq 0} Q^r \exp \left( t_1 \frac{1-Pq^r}{1-q} \right) J^{Y_{-1,-1}}(0,q,Q),
\end{split}
\]
where 
\[
    t^* = t + (1-PT)^2(E_1(t_1,q,Q)) + (1-P)(1-PT)^2(E_2(t_1,q,Q))
\]
with
\[
\begin{split}
    E_1(t_1,q,Q) & \in F_1(t_1,Q) + (1-q) \cdot  \mathbb{Q}[t_1,q^{-1},q][\![Q]\!];
    \\
    E_2(t_1,q,Q) & \in \mathbb{Q}[t_1,q^{-1},q][\![Q]\!].
\end{split}
\]
Then we consider the explicit reconstruction
\[
\begin{split}
    &J^{Y_{-1,-1}}(t, q, Q) =  \sum_{r\geq 0}  J_r^{Y_{-1,-1}}(t^*,q,Q) Q^r  
    \\ 
    & \cdot\exp \left( \frac{ \sum_{i=0}^1\delta_i(t_1,Q)(1-Pq^{r})(1-PT q^r)^i + \sum_{i=0}^3 \e_i(t_1,Q) (1-PT q^{r})^i}{(1-q)} \right) 
    \\
    &\quad \cdot\left( \sum_{i=0}^1 s_i(t_1,q,Q)(1-Pq^r)(1-PT q^r)^i + \sum_{i=0}^3 u_i(t_1,q,Q) (1-PTq^r)^i \right),
\end{split}
\]
for some uniquely determined $\delta_i(t_1,Q)$, $\epsilon_i(t_1,Q)$, $s_i(t_1,q,Q)$, and $u_i(t_1,q,Q)$, where
\[
\begin{split}
    \epsilon_2(t_1,Q) &= - F_1(t_1,Q), \quad u_0(t_1,q,Q) = 1,
    \\
    u_1(t_1,q,Q) &=  \epsilon_i(t_1,Q)=\delta_i(t_1,Q)=s_i(t_1,q,Q)=0\quad \forall i=0,1.
\end{split}
\]
Note that for any choice of $\epsilon_3(t_1,Q)$, $u_2(t_1,q,Q)$, and $u_3(t_1,q,Q)$ will not change the $\mathcal{K}_-$ part. The proposition follows from a long and direct computation.
\end{proof}
\begin{remark}
The invariance of $\mathcal{K}_-$ part in the proof of Proposition~\ref{prop_multicover_t=0} is not surprising. Any choice of $u_2(q,Q)$ and $u_3(q,Q)$ corresponds to the input lies in $(1-q)\cdot H^{\geq 4}(Y_{-1,-1})[\![q, Q]\!]$ and hence it gives no contribution by Lemma~\ref{lemma_vanishingdeg2}.

The same argument works for Proposition~\ref{prop_multicover_t} with one addition. Let $\mathbf{t} \in  t_1(1-P) + \sum_{i=2}^3\epsilon_i(t_1,Q)(1-PT)^i + (1-q) \cdot H^{\geq 4}[t_1,q^{-1},q][\![Q]\!] $. We have
\[
\begin{split}
\sum_{n\geq 0}\sum_{r\geq 0} & \frac{Q^r}{n!}\langle \frac{1}{1-qL}, \mathbf{t},\dots, \mathbf{t}\rangle_{0,n+1,r}^{Y_{-1,-1}}
 =\frac{ t_1 \epsilon_2(t_1,Q)}{1-q}.
\end{split}
\]
by Lemma~\ref{lemma_vanishingdeg2}. Note that all terms equal to zero except the case when $n=2$ and $r=0$. In that case, the contribution only involves $\epsilon_2(t_1,Q)$.
\end{remark}

The following corollary was first obtained by Jockers and Mayr.

\begin{corollary}[Multiple cover formula, {\cite[\S~2.3]{Jockers_Mayr_2019} } ] \label{c:4.8}
For $t = t_1(1-P)$, the small $J$-function for $X_{-1,-1}$ with input $0$ and $t$ are:
\[
\begin{split}
\frac{1}{1-q} & J^{X_{-1,-1}}(0,q,Q) = 1 + \Big(1 +(1-P)\Big)\sum_{r \geq 1} \, Q^r a(r, q^r) 
\\
& + (1-P) \sum_{r\geq 1} Q^r \, b(r, q^r) ;
\\
\frac{1}{1-q} & J^{X_{-1,-1}}(t, q, Q) = 1 + \frac{t_1(1-P)}{1-q}
\\
&+  \Big(1 +(1-P)\Big)\sum_{r \geq 1} \, Q^r\Big( a(r, q^r) + c(r,q, t_1) \Big) 
\\
& + (1-P) \sum_{r\geq 1} Q^r \, \Big( b(r, q^r) + d(r,q, t_1) \Big) .
\end{split}
\]
\end{corollary}

\begin{proof}
Since the zero section $P^1$ has normal bundle $\O(-1) \oplus \O(-1)$, the quantum $K$-invariants of $r \ell$ in $X_{-1,-1}$ are exactly the same as those in $Y_{-1,-1}$. The only difference in $J$-functions comes from different bases of the $K$-groups and the Poincar\'e pairings. The net result is the removal of the factor $(1-PT)^2$ from the specialized $J^{Y_{-1,-1}}(q, Q)$ for non-zero degree terms in Propositon~\ref{prop_multicover_t=0} and \ref{prop_multicover_t}.
\end{proof}

\subsection{Virtual Clemens conjecture}

We now give a heuristic derivation and an interpretation of the relationship between Gopakumar--Vafa invariants and quantum $K$-invariants at genus zero by a multiple cover formula. This has served to guide us in our search for the current formulation of 
Theorem~\ref{conjecture_npt}, even though the actually proof follows a completely different approach. Of course, the original formulations of Jockers--Mayr \cite{Jockers_Mayr_2019} and Garoufalidis--Scheidegger \cite{Garoufalidis_Scheidegger_2022} have been helpful too.

The following ``folklore conjecture'' in enumerative geometry is our starting point.

\begin{conjecture}[Virtual Clemens conjecture]
For the purpose of computing genus zero enumerative invariants on Calabi--Yau threefolds $X$, one may assume that there are only finitely many isolated rational curves $\{C_i \subset X \}$. Furthermore, $C_i$ are all smooth $(-1,-1)$ curves. 
\end{conjecture}

Assume that we are given an ``ideal'' Calabi--Yau threefold $X$ satisfying the above conjecture, by Lemma~\ref{l:4.4}, each isolated $(-1,-1)$-curve (in any degree $\beta$) contributes $1$ to the Gopakumar--Vafa invariants, independently of $\beta$. Therefore, there are $\GV_{0,\beta}$ isolated $(-1,-1)$-curves in degree $\beta$. For each of these isolated curves, quantum $K$-theory allows multiple $r$-covers of the isolated $(-1,-1)$-curve. The coefficients $a(r, q^r)$, $b(r,q^r)$, $c(r,q,t_1)$ and $d(r,q,t_1)$ of the $r$-covers come from the $J$-function of $X_{-1,-1}$. The only addition is the factor of $\beta_j:=\int_{\beta} D_j$ with $D_j = \ch(\Phi^{1j})$, which comes from the divisor axiom.

In summary, the Virtual Clemens Conjecture and the multiple cover formula imply that GV is equivalent to QK for all Calabi--Yau threefolds in genus zero.

\bibliographystyle{alpha}

\bibliography{zbib}
    
\end{document}